\documentclass[reqno]{amsart}
\usepackage[T1]{fontenc}
\usepackage[utf8]{inputenc}
\usepackage{amssymb,wasysym}
\usepackage{amsfonts}
\usepackage{amsmath,mathtools}
\usepackage{mathrsfs}
\usepackage{graphicx}
\usepackage{subcaption}
\usepackage{hyperref}
\usepackage{esint}
\usepackage{verbatim}

\usepackage{bbm}
\usepackage[usenames,dvipsnames]{xcolor}
\usepackage{enumitem}
\usepackage{setspace}
\usepackage{stackrel}

\newcommand{\M}[1]{\mathbb{M}^{#1}}
\newcommand{\N}[1]{\mathbb{N}^{#1}}
\newcommand{\R}[1]{\mathbb{R}^{#1}}

\newcommand{\bb}{\boldsymbol{d}}
\newcommand{\be}{\boldsymbol{e}}
\newcommand{\bE}{\boldsymbol{E}}
\newcommand{\bef}{\boldsymbol{f}}
\newcommand{\bF}{\boldsymbol{F}}
\newcommand{\bg}{\boldsymbol{g}}
\newcommand{\bh}{\boldsymbol{h}}
\newcommand{\bm}{\boldsymbol{m}}
\newcommand{\bM}{\boldsymbol{M}}
\newcommand{\bn}{\boldsymbol{n}}
\newcommand{\bs}{\boldsymbol{s}}
\newcommand{\bt}{\boldsymbol{t}}

\newcommand{\bu}{\boldsymbol{u}}
\newcommand{\bv}{\boldsymbol{v}}

\newcommand{\bx}{\boldsymbol{x}}

\newcommand{\bzero}{\boldsymbol{0}}

\newcommand{\cC}{\mathcal C}

\newcommand{\cF}{\mathcal F}

\newcommand{\cM}{\mathcal M}
\newcommand{\cO}{\mathcal O}

\newcommand{\cR}{\mathcal R}
\newcommand{\cS}{\mathcal S}

\newcommand{\cU}{\mathcal U}
\newcommand{\cV}{\mathcal V}

\newcommand{\cZ}{\mathcal Z}

\newcommand{\fg}{\mathfrak g}
\newcommand{\fh}{\mathfrak h}

\newcommand{\fz}{\mathfrak z}

\newcommand{\de}{\mathrm d}

\newcommand{\SE}[1]{\mathrm{SE}(#1)}

\providecommand{\abs}[1]{\left\lvert#1\right\rvert}

\newcommand{\sign}{\operatorname{sign}}

\DeclareMathOperator{\spann}{span}
\DeclareMathOperator{\rank}{rank}

\renewcommand{\geq}{\geqslant}
\renewcommand{\leq}{\leqslant}

\newcommand{\average}{{\mathchoice {\kern1ex\vcenter{\hrule
height.4pt width 8pt depth0pt}
\kern-11pt} {\kern1ex\vcenter{\hrule height.4pt width 4.3pt
depth0pt} \kern-7pt} {} {} }}

\mathchardef\emptyset="001F

\providecommand{\U}[1]{\protect\rule{.1in}{.1in}}
\numberwithin{equation}{section}
\newtheorem{definition}{Definition}[section]
\newtheorem{theorem}[definition]{Theorem}

\newtheorem{proposition}[definition]{Proposition}

\newtheorem{corollary}[definition]{Corollary}

\newtheorem{example}[definition]{Example}

\theoremstyle{definition} {\newtheorem{remark}[definition]{Remark}}

\usepackage{pgf,tikz}
\usetikzlibrary{arrows}
\usetikzlibrary{decorations.markings,decorations.pathmorphing,patterns,shapes}

\makeindex

\title[Gait controllability of length-changing  slender microswimmers]{Gait controllability of length-changing slender microswimmers}%
\author[P.~Gidoni]{Paolo Gidoni}
\address[P.~Gidoni]{Dipartimento Politecnico di Ingegneria e Architettura, Università  degli Studi di Udine, Via delle Scienze 206, 33100 Udine, Italy}
\email{paolo.gidoni@uniud.it}
\author[M.~Morandotti]{Marco Morandotti}
\address[M.~Morandotti]{Dipartimento di Scienze Matematiche ``G.~L.~Lagrange'', Politecnico di Torino, Corso Duca degli Abruzzi, 24, 10129 Torino, Italy}
\email{marco.morandotti@polito.it}
\author[M.~Zoppello]{Marta Zoppello}
\address[M.~Zoppello]{Dipartimento di Scienze Matematiche ``G.~L.~Lagrange'', Politecnico di Torino, Corso Duca degli Abruzzi, 24, 10129 Torino, Italy}
\email{marta.zoppello@polito.it}

\date{\today}


\makeindex

\begin{document}

\begin{abstract}
Controllability results of four models of two-link microscale swimmers that are able to change the length of their links are obtained. The problems are formulated in the framework of Geometric Control Theory, within which the notions of fiber, total, and gait controllability are presented, together with sufficient conditions for the latter two.
The dynamics of a general two-link swimmer is described by resorting to Resistive Force Theory and different mechanisms to produce a length-change in the links, namely, active deformation, a sliding hinge, growth at the tip, and telescopic links. Total controllability is proved via gait controllability in all four cases, and illustrated with the aid of numerical simulations.
\end{abstract}

\maketitle



\section{Introduction}

	In the landscape of mathematical models for swimming micro-organisms, much attention is devoted to slender swimmers, since many living beings have a shape that can be approximated by a filament with sufficient accuracy (including, for instance, \emph{C.~elegans} and spermatozoa).
	Models for swimming are particularly useful for the design of Nature-inspired robots, which necessarily undergo simplifications and reductions in the available degrees of freedom. A relevant issue is therefore the search for \emph{minimal} models, involving a small number of parameters, but still describing the essential features of the swimmer, including its ability to successfully move and reach a given final position.
	
	Clearly, an oversimplification might be deleterious to the swimmer capabilities; such risk is even more compelling at microscale, due to time-reversibility of low Reynolds number fluidodynamics.
	The celebrated example of Purcell's Scallop Theorem \cite{Purcell77} addresses precisely this situation: a ``scallop'' consists of two equal segments hinged together at one of their extremities, with the ability to change in time the angle between the segments (this models the opening and closing its ``valves'').  If  this is an effective strategy for rectilinear locomotion when inertial forces are relevant in the dynamics of the fluid, such a scallop cannot swim at low Reynolds number, where viscous forces are much stronger than inertial ones and the flow is time-reversible: after any periodic opening and closing of the valves, the scallop returns to its initial position.
	
	The lesson of the Scallop Theorem is that the symmetry of the reciprocal motion of opening and closing the valves must be broken to achieve non-zero net displacement. A first solution was provided by Purcell's swimmer~\cite{Purcell77}, consisting of three concatenated links, with the ability to change independently the angles between each couple of adjacent links. If two independent controlled parameters are enough to avoid reciprocal motion and properly change position, the question becomes whether they are also sufficient to reach any desired position.
	
Before continuing, let us clarify our framework. By \emph{swimming}, we intend the ability to achieve a non-zero displacement from an initial position by changing shape, only relying on hydrodynamic reaction forces. In this paper, the latter will be described through Resistive Force Theory, which is a valid approximation of Stokes equations for a slender filament~\cite{GH1955,Lauga_book}.  
When building a model for swimming, it is therefore helpful to distinguish between \emph{shape variables}~$s$ and \emph{position variables}~$g$: the state of the swimmer is then described by the pair $(s,g)\in\cZ$, where $\cZ$ is a suitable manifold. For example, for planar locomotion of Purcell's swimmer, the center of the middle link and its orientation are the position variables and the angles that the lateral links form with the middle one are the two shape parameters. 
Swimming consists of obtaining a change in the~$g$ variables as a consequence of a change in the~$s$ variables. If the shape change is periodic, it is usually referred to as a \emph{gait}. 
Locomotion can be therefore studied as a \emph{controllability} problem, namely, establishing whether it is possible to reach a given final state from a given initial one, an issue  successfully addressed by a discipline called Geometric Control Theory (see~\cite{LibroAgrachev,LibroCoron,LibroJurdjevic} and the references therein for a comprehensive introduction).

In this setting, each control parameter corresponds to a vector field on the state manifold $\cZ$. Remarkably, to achieve controllability it is not necessary to span all possible directions of the tangent bundle of $\cZ$ with the control vector fields. Indeed, minimal swimmers are endowed of fewer control parameters than the dimension of the state manifold $\cZ$: controls are naturally associated with rates of change of the shape variables, which are fewer than the dimension of the state manifold. As we will see in our examples, often the minimal number of control parameters is even less than the position variables $g$.
The reason is that a particular concatenation of motions in the available directions might result in a net motion along a direction which is not directly accessible. 
The formal way to state this is that the Lie brackets of the available vector fields for the motion generate the missing directions and allow the swimmer to move to the desired final state.

In this framework, Purcell's Scallop Theorem is immediate: since the only shape variable is the angle and since the Lie bracket of a vector field with itself vanishes, no non-zero net motion is possible for an idealized scallop performing a reciprocal motion. 
On the contrary, the Lie brackets between the two control vector fields associated with the two (velocities of the) angles in the three-link swimmer do not vanish and allow it to displace.

It is now clear that symmetry breaking is intimately related to the event that the control vector fields available for the motion and their Lie brackets (possibly of higher order, i.e., nested Lie brackets) provide a sufficiently rich ensemble of linearly independent vectors in the ambient space.
To obtain this, there are various strategies, such as adding more shape parameters that can be actuated independently, so each of them is related to a control vector field: this is the case of Purcell's three link swimmer \cite{Purcell77}, of the $N$-link swimmer \cite{ADSGZ,GMZ,MMSZ}), and when the hydrodynamic interaction between two (possibly non-controllable) swimmers is exploited \cite{AZN,MKL,ZMBG}. 
Other symmetry-breaking mechanisms are based instead on hysteresis phenomena occurring during large deformations, see, e.g.,~\cite{BM2024,CDN2023,KoensAl,MourranAl17,MourranAl21, RehorAl21}.

In this paper, we propose four models for an enhanced two-link swimmer to achieve controllability, thus overcoming the Scallop Theorem in the context of microswimming.
To present our results, we discuss different concepts of controllability, such as total controllability (to prescribe the initial and final positions and shapes) and fiber controllability (to prescribe the initial position and shape and the final position). In particular, we focus on the notion of \emph{gait controllability}, an intermediate one (in terms of strength), which we argue to be the most apt for the task. With gait controllability (see Definition~\ref{PC} below) we mean that we prescribe the initial and final position and the \emph{same} initial and final shape. In this way, the feature of a periodic shape change becomes embedded in the structure of the problem, resulting in a more fitting treatment of some situations, see Remarks~\ref{rem:GCTC} and~\ref{rem:GCFC}. Moreover, on a more technical perspective, gait controllability is usually the gateway through which the stronger total controllability is proven, making an independent notion even more useful.

Classical models of slender microswimmers are based on the capability to control curvature or angles along the swimmer body: in addition to the above-mentioned $N$-link swimmers, we recall travelling waves \cite{Taylor1951} or the rotation of a corkscrew flagellum \cite{KTSKS,OAR,SG2012}. The possibility for the swimmer to change the size of some body parts has been considered, for example, for a volume change in a model of two-sphere swimmer \cite{AKO05,ECS2023,Wang2019};  here we study instead the elongation and shortening of a slender filament.
More precisely, we enhance the basic ``non-controllable'' two-link swimmer by enabling it to change the lengths of each of its links, thus endowing it with stronger controllability properties. 

While a change in curvature has a trivial interpretation in terms of relative displacement of the material points of the swimmer, the same is no longer true for a length-change in a filament, which can be produced, for example, by a stretching of the filament, or by tip-localized growth, or even by the protrusion of a telescopic element. A different relative motion of the material points of the swimmers implies that also their velocities are different (both the absolute and the relative ones), leading to different resulting viscous forces associated with the motion in the fluid. In order to highlight the role of the material structure on the dynamics of the swimmer, we consider four alternative mechanisms to produce a length-change in the two links of a two-link swimmer.

The first length-change mechanism, \emph{stretching links} (Section~\ref{sec:stretch}), consists of an active longitudinal strain along each link of the swimmer. This mechanism is inspired by soft active materials, such as hydrogels \cite{MourranAl17,SharanAl21} and magnetosensitive elastomers \cite{HLMS}, which have been receiving an increasing attention in the design of microscale robots \cite{MedAl}. An active strain can be employed to produce a shape change in many ways, such as a shrinking/swelling of the body or variations of the natural curvature in bilayer materials or in coated hydrogels \cite{MourranAl17,MourranAl21}.
Here we consider one of the simplest possibilities, namely a uniform longitudinal deformation of a straight slender filament, producing a change in the length of the link.
Notice however that, in general, the locomotion in a fluid of microbots made of active materials might involve other phenomena in addition to shape-change-based swimming, such as jet propulsion \cite{TanAl} or complex interaction with a substrate when moving in planar environment \cite{RehorAl21,KropacekAl},
which we do not include in our investigation.

The other three mechanisms considered in our paper are inspired by other shape-change strategies employed in robotics, although not yet at the microscopic scale.
The \emph{sliding links} model of Section~\ref{sec:sliding} is inspired by isoperimetric robots~\cite{UHSOHF}. 
The sum of the lengths of the two links is constant and shape change is produced by the filament sliding through the hinge. 
Although the macroscopic effect is a change in the length of each link, this mechanism can be more properly seen as the limit case of a filament with active curvature, with the opening and closing of the two links being produced by changing the curvature near a point in the filament, adding to the device the possibility to change the point of activation smoothly in time. 

The \emph{growing links} model of Section~\ref{sec:growing} is inspired by vine robots~\cite{BCHOH} and root robots~\cite{SMM17}. There is no active deformation along each link, neither in length nor in curvature, and elongation is produced by new material points appearing at the tips of the filament. 
Notice that this does not imply that a new part of the body is generated: for instance, one can consider, as in vine robots, a thin tubular filament  with the yet-to-appear section being folded inside the visible one and being gradually everted during elongation. Since, for our purposes, only points in contact with the fluid are relevant, eversion is equivalent to growth of the body.  Shortening happens in the reverse way, with the material points at the tip gradually being retracted and thus disappearing from our description of the model. 

The growing links model serves also to introduce our last mechanism, the \emph{telescopic links} model of Section~\ref{sec:telescopic}, in which each link works as a telescopic pole with two sections one sliding inside the other one, with the outer section being the one closer to the hinge. 

The points of the inner section which, at a given time, are located inside the outer section are not in contact with the fluid (and do not affect the shape of the swimmer), so they are not accounted for in our description of the body at that time. As in the previous model, therefore, elongation corresponds to a change in the reference configuration, with more points being added to the body. The difference is where such points are added. In the \emph{growing links} model this happened at the tip of the link; in the \emph{telescopic links} model instead it is mid-link, at the joint between the two sections,  with new material points being added to the inner section on its junction-end, as it is protruded from the outer one.

The paper is organized as follows: after reviewing Resistive Force Theory and the bases of Geometric Control Theory, we introduce and compare the different notions of controllability.  We give a sufficient condition for gait controllability (Theorem~\ref{th:gaitcontrollable}) and, in Proposition~\ref{PCimpliesTC}, we prove that gait controllability and total controllability of the shape imply total controllability of the whole system. 
In Section~\ref{sec:generalmodel}, we introduce a general mechanical description of the swimmer and interpret it as a control system.
In Section~\ref{sec:esempi}, we present the three (plus one) models of a two-link swimmer with length-changing links, and prove the gait and total controllability of each of them.
Finally, in Section~\ref{sec_num_disc}, we include some numerical simulations showing concatenations of Lie brackets that generate both translation and rotation of the swimmers, and offer a summary of the research.


\section{Gait controllability} \label{sec:control}
We study the controllability of locomotion systems whose state is identified by a variable $s$ describing the (time-dependent) shape of the locomotor and a variable $g$ describing its position. Thus, we take $s\in \cS$, where the shape space $\cS$ is a suitable parallelizable $n$-dimensional manifold, and $g\in G$, where the position space $G$ is a connected $d$-dimensional Lie group with corresponding Lie algebra $\fg$ of dimension~$d$, whose elements are represented by $d\times d$ matrices. We recall that an $n$-dimensional manifold is parallelizable if there exist~$n$ smooth vector fields on the manifold defining a basis of the tangent space at each point.
 
The dynamics of the system is assumed to be of the form
\begin{equation}\label{system}
	\begin{pmatrix}
	\dot s\\\dot g\end{pmatrix}=
	\begin{pmatrix}
	f(s,u)\\g\xi(s,u)\end{pmatrix}\qquad\text{with}\quad f(s,u)=F(s)u\quad\text{and}\quad \xi(s,u)=\sum_{i=1}^m \xi_i(s)u_i\,,
\end{equation}
where $F\colon \cS\to\R{n\times m}$ 
and $\xi_i\colon\cS\to\fg$ are nonlinear 
smooth functions and $u=(u_1,\dots,u_m)$  is the vector of the controls. We denote with $F_i(s)\in T_s\cS$ the columns of the matrix $F(s)$, where $T_s\cS\simeq \R{n}$ is the tangent space to $\cS$ at $s$.
The controlled ODE \eqref{system} can be written in matrix form
\begin{equation}\label{system_matrix}
	\dot z=Z(z) u=\sum_{i=1}^m Z_i(s,g) u_i\,,
\end{equation}
where we have posed $z=(s,g)^\top\in\cS\times G\eqqcolon\cZ$, and where $Z_i$ is the $i$-th column of the matrix $Z$. 
Notice that, upon isomorphism, the element $g\in G$ can be identified with a vector in~$\R{d}$, so that $Z\in\R{(n+d)\times m}$ takes the form 
$$Z=\begin{pmatrix}
	Z^{\cS}\\ Z^G
\end{pmatrix},$$
where the $i$-th column $Z^{\cS}_i$ is $F_i$ and the $i$-th column $Z^G_i$ is the vector representation of the matrix~$g\xi_i$\,.
Owing to the smoothness of the maps~$F_i$ and~$\xi_i$\,, by classical results (see, e.g., \cite{Hale}), there exists a unique solution to the Cauchy problem associated with \eqref{system_matrix}, namely
\begin{equation}\label{system_matrix_CP}
	\begin{cases}
	\dot z=Z(z)u,\\
	z(0)=z^\circ,
	\end{cases}
\end{equation}
for every initial condition $z^\circ=(s^\circ,g^\circ)\in\R{n+d}$, provided that the control maps $t\mapsto u(t)$ are measurable.\\
For systems as \eqref{system}, we recall the notion of equivariance.

\begin{definition}\label{equiVF}
	Let $g\in G$. 
	A vector field $X$ on $\cZ$ is said to be \emph{equivariant} with respect to the group action 
	\begin{equation}\label{eq:groupaction}
		\Psi_g\colon \cZ \to \cZ, \qquad z=(s,g')^\top\mapsto \Psi_g(z)\coloneqq (s,gg')^\top
	\end{equation}
if, denoting by $(\cdot)_*$ the push-forward, it holds
\begin{equation*}
	(\Psi_g)_*X(z)=X(\Psi_g (z)).
\end{equation*}
\end{definition}
In practical terms, a vector field is equivariant if its $G$-component $X^G$ can be expressed in the specific factored form $X^G=g\Xi(s)$, that is, the left-translation by $g$ of $\Xi(s)$.

From the point of view of control theory, it is crucial to identify suitable notions of controllability, loosely speaking the ability to steer the system between two given initial and final configurations.
The most classical notion involves moving the whole state of the system $z=(s,g)$ from any given initial state~$z^\circ$ to any given final state~$z^\bullet$.
\begin{definition}[total controllability]\label{TC}
	System \eqref{system} is \emph{totally controllable} if for any given initial and final configurations $z^\circ$ and $z^\bullet$ there exist a time $T>0$ and a control function $u\in L^\infty([0,T];\R{m})$ such that the unique solution to \eqref{system_matrix_CP} with control~$u$ satisfies $z(T)=z^\bullet$.
\end{definition}
In many relevant models, total controllability turns out to be too strict of a notion, so that the weaker concept of fiber controllability has been introduced \cite{KM}. 
\begin{definition}[fiber controllability]\label{FC}
	System \eqref{system} is \emph{fiber controllable} if for any given initial configuration $z^\circ=(s^\circ,g^\circ)^\top$ and any given final position $g^\bullet$ there exist a time $T>0$ and a control function $u\in L^\infty([0,T];\R{m})$ such that the unique solution to \eqref{system_matrix_CP} with control~$u$ satisfies $z(T)=(s(T),g(T))^\top=(s(T),g^\bullet)^\top$.
\end{definition}
We emphasize that no condition is imposed on the final shape~$s(T)$; this is somehow more natural for locomotion models, where the main scope is to displace regardless of final shape assumed at the arrival point.
Still within the realm of locomotion models, it is observed in nature that most living beings use periodic shape changes to move and this inspires the programming of many robotic devices. 
Moreover, it has been proved in \cite{GJ2020} that among the solutions to a certain optimal control problem for systems like \eqref{system} with two controls (that is, $m=2$)  there is a periodic one.
This motivates the following definition.
\begin{definition}[gait controllability]\label{PC}
	System \eqref{system} is \emph{gait controllable} if for any given initial configuration $z^\circ=(s^\circ,g^\circ)^\top$ and any given final position $g^\bullet$ there exist a time $T>0$ and a control function $u\in L^\infty([0,T];\R{m})$ such that the unique solution to \eqref{system_matrix_CP} with control~$u$ satisfies  $z(T)=(s(T),g(T))^\top=(s^\circ,g^\bullet)^\top$.
\end{definition}
For our purposes we introduce also the following pointwise version of gait controllability.
\begin{definition}[gait controllability at $s^\star$]\label{PCs}
	System \eqref{system} is \emph{gait controllable at a given shape $s^\star\in \cS$}  if for any given initial and final positions  $g^\circ,g^\bullet \in G$ there exist a time $T>0$ and a control function $u\in L^\infty([0,T];\R{m})$ such that the unique solution to \eqref{system_matrix_CP} with initial condition $z^\circ=(s^\star,g^\circ)^\top$ and control~$u$ satisfies  $z(T)=(s(T),g(T))^\top=(s^\star,g^\bullet)^\top$.
\end{definition}
System \eqref{system} is gait controllable if and only if it is gait controllable at every $s^\star\in \cS$.

The proof of the following statement is an immediate consequence of the definitions just given.
\begin{proposition}\label{TCtoPCtoFC}
	Total controllability $\Rightarrow$ gait controllability $\Rightarrow$ fiber controllability.
\end{proposition}
In general, the converse implications do not hold true, as we will show in Examples \ref{ex:notPCtoTC} and \ref{ex:notFCtoPC} below. Yet, it is possible to identify a suitable sufficient condition for the equivalence of total and gait controllability.

\begin{proposition}\label{PCimpliesTC}
	Let us assume that system \eqref{system} is gait controllable at a given shape $s^\star$ and that the shape-reduced system $\dot s=f(s,u)= F(s)u$ is totally controllable. Then also system \eqref{system} is totally controllable.
\end{proposition}
\begin{proof}
Given an initial configuration $z^\circ=(s^\circ,g^\circ)^\top$, by the total contollability of the shape-reduced system there exist a time $T_a$ and a control function $u_a\in L^\infty([0,T_a];\R{m})$  such that the unique solution to \eqref{system_matrix_CP} with control~$u_a$ satisfies $z(T_a)=(s^\star,g_a)^\top$ for some $g_a\coloneqq g(T_a)\in G$. 

Analogously, for a desired final configuration $z^\bullet=(s^\bullet,g^\bullet)^\top$, there exist a time $T_c$ and a control function $u_c\in L^\infty([0,T_c];\R{m})$  such that the unique solution to \eqref{system_matrix} with initial condition $(s^\star,e)^\top$ with control~$u_c$ satisfies $z(T_c)=(s^\bullet,g_c))^\top$ for some $g_c\coloneqq g(T_c)\in G$. By the equivariance of the vector fields (see Definition~\ref{equiVF}), it follows that the unique solution to \eqref{system_matrix} with initial condition $(s^\star,g_c^{-1}g^\bullet)^\top$ with control~$u_c$ satisfies $z(T_c)=(s^\bullet,g^\bullet))^\top=z^\bullet$. 

Finally, by the gait controllability at $s^\star$, there exist a time $T_b>0$ and a control function $u_b\in L^\infty([0,T_b];\R{m})$ such that the unique solution to \eqref{system_matrix} with initial condition $(s^\star,g_a)^\top$ and control~$u_b$ satisfies $ z(T_b)=(s^\star,g_c^{-1}g^\bullet)^\top$. The concatenation $u\coloneqq {u_a}^\frown {u_b}^\frown u_c$ of the three control functions steers the system from $z^\circ$ to $z^\bullet$.
\end{proof}

We now present some tools and results from Geometric Control Theory, which are useful to prove the controllability notions just introduced.
Thanks to the expressions in \eqref{system}, it is evident that, in system \eqref{system_matrix}, the vector fields $Z_i$ are equivariant with respect to the group action \eqref{eq:groupaction}. Moreover, for all $Z_i,Z_j\in \R{n}\times T_gG$ and for any $g\in G$ the Lie bracket $[Z_i,Z_j]_{T_{(s,g)^\top}\cZ}$ is equivariant with respect to the group action. 
Here, the Lie bracket between vectors in $T_{(s,g)^\top}\cZ$ is defined as
\begin{equation}\label{brackets25}
	[Z_i,Z_j]_{T_{(s,g)^\top}\cZ} 
\coloneqq\begin{pmatrix}
[F_i,F_j]_{T_s\cS}\\
g\Bigl(\, [\xi_i,\xi_j]_\fg + \nabla_s\xi_j\cdot F_i - \nabla_s\xi_i\cdot F_j \,\Bigr)
\end{pmatrix}=
\Psi_g\big(\big[\widetilde Z_i,\widetilde Z_j\big]_{T_{(s,e)^\top}\cZ}\big),
\end{equation}
where $[F_i,F_j]_{T_s\cS}=[F_i,F_j]_{\R{n}}\coloneqq (\nabla_s F_j)F_i - (\nabla_s F_i)F_j$ and $[\xi_i,\xi_j]_\fg\coloneqq \xi_i\xi_j-\xi_j\xi_i$ is the matrix commutator between the elements $\xi_i,\xi_j\in\fg$, while by $\nabla_s\xi\cdot f$ we denote the contraction of the $(d\times d\times n)$-tensor $\nabla_s \xi$ with the $n$-vector $f$ resulting in a $(d\times d)$-matrix and $\widetilde Z=Z(s,e)=(f(s),\xi(s))^\top$ (see \eqref{system_matrix}).
When no risk of misunderstanding occurs, we will drop the subscripts on the brackets and simply use the symbol~$[\cdot,\cdot]$.
It is useful to introduce the projection on the tangent space $T_gG$, which we denote by 
\begin{equation}\label{Pi_g}
\Pi_g^G\colon \R{n}\times T_gG\to T_gG,
\end{equation} 
which is useful to define some subspaces of the Lie algebra~$\fg$. 
Indeed, observing that $\Pi^G_g[Z_i,Z_j]= g\bigl([\xi_i,\xi_j]_\fg + \nabla_s\xi_j\cdot F_i - \nabla_s\xi_i\cdot F_j \bigr)$ and that $T_eG=\fg$, we introduce the following subspaces of~$\fg$:
\begin{equation*}
	\fh_k(s)=\Pi^G_e(H_k(s)) \qquad \text{for $k\geq 1$}\,,
\end{equation*}
where the subspaces $H_k(s)$ of $T_{(s,e)^\top}\cZ$ are defined as
\begin{align*}
	H_1(s) \coloneqq &\, \spann\{\widetilde Z_i:i=1,\ldots,m\}\,, \\
	H_2(s) \coloneqq &\, \spann\bigl\{\big[\widetilde Z_i,\widetilde Z_j\big] \quad\text{for every $\widetilde Z_i,\widetilde Z_j\in \R{n}\times \fg$}\bigr\}\,,\\
	H_3(s) \coloneqq &\, \spann\bigl\{\big[\widetilde Z,h\big] \quad\text{for every $\widetilde Z\in \R{n}\times \fg$ and $h\in H_2(s)$}\bigr\}\,,\\
	\vdots\,\,&\\
	H_k(s) \coloneqq &\, \spann\bigl\{\big[\widetilde Z,h\big] \quad\text{for every $\widetilde Z\in \R{n}\times \fg$ and $h\in H_{k-1}(s)$}\bigr\}\,.
\end{align*}
We denote by $\fz_z$ the Lie algebra generated by the vector fields  $Z_i\in \R{n}\times T_gG$, for $i=1,\dots, m$, so that $\fz_{(s,e)^\top}$ is the Lie algebra generated by the vector fields $\widetilde Z_i$ and therefore the $H_k(s)$ are subspaces of $\fz_{(s,e)^\top}$. 

The following theorems provide sufficient conditions for each of the  controllability notions introduced above. The first one presents two instances of the Chow--Rashevskii Theorem in the case the vector fields are of class $\cC^\infty$ (see \cite[Theorem~5.9]{LibroAgrachev}) or analytic (see \cite[Corollary~5.17]{LibroAgrachev}).
\begin{theorem}\label{RC}
Let $Z_i\in \R{n}\times T_gG$ for $i=1,\dots, m$ be a family of $\cC^\infty$ vector fields on~$\cZ$. 
Then system \eqref{system} is totally controllable if
\begin{equation}\label{dimzg}
\dim \fz_z=n+d\qquad \text{for every $z\in \cZ$,} 
\end{equation}
that is the dimension of the Lie algebra~$\fz_z$ equals that of the tangent space $\R{n}\times T_gG$, for every $g\in G$.

Moreover, if the vector fields $Z_i$'s are real analytic, then condition \eqref{dimzg} is also necessary for total controllability.
\end{theorem}
Condition \eqref{dimzg} is also known in the literature as the Lie algebra $\fz_z$ being fully generated.
If all the vector fields $Z_i$'s are equivariant, we have the following corollary.
\begin{corollary}\label{RCcor}
Let $Z_i\in \R{n}\times T_gG$ for $i=1,\dots, m$ be a family of equivariant $\cC^\infty$ vector fields on~$\cZ$. 
Then system \eqref{system} is totally controllable if $\dim \fz_{(s,e)}=n+d$, for every $s\in\cS$.

Moreover, if the vector fields $Z_i$'s are real analytic, then this condition is also necessary for total controllability.
\end{corollary}
\begin{proof}
Since linear independence is preserved by equivariance, we deduce that $\dim\fz_z=\dim\fz_{(s,e)^\top}$ for every $z\in \cZ$.
\end{proof}

The following theorem, which is the theoretical cornerstone of this paper, provides a sufficient condition for gait controllability.
\begin{theorem}\label{th:gaitcontrollable}
Let $\widetilde Z_i\in \R{n}\times \fg$ for $i=1,\dots, m$ be a family of equivariant $\cC^\infty$ vector fields on~$\cZ$ and assume that the matrix~$F$ is constant in~$s$.
Then system \eqref{system} is gait controllable at a given shape $s^\star\in \cS$ if 
\begin{equation}\label{abovecondition}
	\fg= \fh_2(s^\star)\oplus \fh_3(s^\star)\oplus \cdots \oplus \fh_{k}(s^\star)\oplus \cdots .
\end{equation}	
Moreover, if \eqref{abovecondition} holds true for every $s^\star\in\cS$, then the system is gait controllable according to Definition~\ref{PC}.
\end{theorem}

The proof of Theorem \ref{th:gaitcontrollable} will be provided below, after some preliminary observations.
First, we recall the celebrated Orbit Theorem. 
Given a family $\cF$ of vector fields on a manifold $\cM$, we define the orbit of the family $\cF$ through the point $p\in \cM$ as
\begin{equation*}
	\cO_p\coloneqq\{p\circ e^{t_1V_1}\circ \cdots \circ e^{t_kV_k} \colon t_i\in \R{}, V_i\in \cF, k\in \N{} \}\subseteq\cM,
\end{equation*}
where $e^{tV}$ is the flow generated by $V$ (cf.~\cite[Section~2.4]{LibroAgrachev}), and $\circ$ denotes the standard composition. In practice, the orbit $\cO_p$ is the set of points in $\cM$ that can be reached starting from $p$ by following a finite collection of vector fields $V_i$\,, each for a period of time $\lvert t_i\rvert$, with the clear meaning that if $t_i<0$ then one follows $-V_i$\,.
\begin{theorem}[{Orbit Theorem, Nagano--Sussmann \cite[Theorem~5.1]{LibroAgrachev}}] \label{th:orbit}
	Let~$\cM$ be a smooth manifold and~$\cF$ be a family of $\cC^\infty$ vector fields on~$\cM$. Then each orbit $\cO_p$ of $\cF$ is a connected immersed submanifold of $\cM$.
	
	Moreover, if the manifold $\cM$ and the vector fields are analytic, then
 the tangent space of $T_q\cO_p$, with $q\in\cO_p$, is given by the Lie algebra generated by $\cF$ at $q$; in particular the dimension of the Lie algebra is constant as $q$ varies on $\cO_p$. 
\end{theorem}

Let us also explicit the structure of the brackets \eqref{brackets25} in terms of the control function~$u$.
In particular, for given $i,j\in\{1,\ldots,m\}$ ($i\neq j$), the bracket $[\widetilde Z_i,\widetilde Z_j]$ is generated by the control $u\colon[0,\tau]\to\R{m}$, for a small $\tau>0$, defined by
\begin{subequations}\label{loop_gen}
\begin{equation}\label{loop}
u(t)=\begin{cases}
Ae_i & \text{for $0\leq t<\tau/4$,}\\
Be_j & \text{for $\tau/4\leq t<\tau/2$,}\\
-Ae_i & \text{for $\tau/2\leq t<3\tau/4$,}\\
-Be_j & \text{for $3\tau/4\leq t\leq\tau$,}\\
\end{cases}\qquad a,b\in\R{},
\end{equation}
where here $e_i$ and $e_j$ are the $i$-th and $j$-th vectors of the canonical basis in $\R{m}$.
The zero-average control above produces a $\tau$-periodic loop of $[0,\tau]\ni t\mapsto \int_0^t u(\tilde t)\,\de \tilde t$, but in general each actuation cycle might produce a non-zero displacement in the state of the system, which is described, up to leading order in $\tau$, by $[\widetilde Z_i,\widetilde Z_j]ab\tau^2$.
In particular, we have $g(\tau)-e=ab\tau^2\Pi_e^G\big(\big[\widetilde Z_i,\widetilde Z_j\big]\big)+o(\tau^2)$, cf.~\cite[Eq.~(3.2)]{LibroCoron}, so that a non-vanishing $\Pi_e^G\big(\big[\widetilde Z_i,\widetilde Z_j\big]\big)$ corresponds, in the limit $\tau\to0$, to a non-zero displacement in the position variable. On the other hand, to have a non-zero displacement it is not sufficient to have a vanishing bracket, since also the higher order terms must be taken into account. For instance, as we show in the proof of Theorem \ref{th:gaitcontrollable}, a zero displacement in the shape variable (and thus its $\tau$-periodicity) can be obtained under the stricter assumptions of constant matrix~$F$.

The very same behavior is observed also for higher order brackets $[\widetilde{Z}_i,h]$ with $h\in H_{k-1}(s)$ for $k\geq 3$. Indeed, by induction on $k$, the corresponding control $u\colon[0,\tau]\to\R{m}$ is defined as
\begin{equation}\label{loopk}
	u(t)=\begin{cases}
		Ae_i & \text{for $0\leq t<\tau/4$,}\\
		Bu_h(4(t-\frac{\tau}{4})) & \text{for $\tau/4\leq t<\tau/2$,}\\
		-Ae_i & \text{for $\tau/2\leq t<3\tau/4$,}\\
		-Bu_h(4(t-\frac{3\tau}{4}))  & \text{for $3\tau/4\leq t\leq\tau$,}\\
	\end{cases}\qquad a,b\in\R{},
\end{equation}
\end{subequations}
where $u_h\colon[0,\tau]\to\R{m}$ is the control function associated with the bracket $h$.

\begin{proof}[Proof of Theorem~\ref{th:gaitcontrollable}]
We start by observing that the hypothesis of constant vector fields $F_i$ implies, after applying the control $u\colon[0,\tau]\to\R{m}$ associated with a bracket $h\in H_k(s^\star)$, $k\geq2$, defined as in \eqref{loop_gen}, that the final shape is exactly the initial one, that is $s(\tau)=s^\star$. Indeed, by \eqref{system}, for $k=2$ we have
\begin{equation}\label{eq:periodicshape}
	s(\tau)-s^\star=\int_0^\tau F u(t) \,\de t=\int_0^{\tau/4}\!\! AF_i\,\de t+\int_{\tau/4}^{\tau/2}\!\! BF_j\,\de t+\int_{\tau/2}^{3\tau/4}\!\! -AF_i\,\de t+\int_{3\tau/4}^{\tau}\!\! -BF_j\,\de t =0 \,.
\end{equation}
For $k>2$, the same property is proven by induction on $k$ by a computation analogous to \eqref{eq:periodicshape}, noticing that the second and fourth integral correspond to the displacement in shape produced by a bracket of order $k-1$ and are therefore zero by induction.

Since arbitrarily small loops of the form \eqref{loop_gen} generate a displacement in the position variable proportional, up to the leading order, to the projection $\Pi^G_e (h)$ of their corresponding  Lie bracket $h\in H_k(s^\star)$, $k\geq 2$, and by hypothesis  $\fg= \fh_2(s^\star)\oplus \fh_3(s^\star)\oplus \cdots \oplus \fh_{k}(s^\star)\oplus \cdots$, such small periodic loops generate any direction of the Lie algebra~$\fg$. In particular we can take $d$ linearly independent vector fields $\widetilde{h}_i\in H_{k(i)}(s^\star)$, with $k\geq 2$ and $i=1,\dots,d$.
By the Orbit Theorem~\ref{th:orbit} applied to the family $\cF=\{\Pi_e^G(\widetilde{h}_1),\dots,\Pi_e^G(\widetilde{h}_d)\}$,  it follows that by finite concatenations of loops we are able to reach any point in an open neighborhood~$\cU_e$ of~$e$. 

Since the vector fields $\widetilde Z_i$ and their brackets are equivariant, this latter property can be transferred by the group action \eqref{eq:groupaction} to every $g\in G$. Thus, from every point $g\in G$ we can reach, by concatenating loops, any point in $\cU_g=g+\cU_e$.

By the connectedness of~$G$, let~$\gamma$ be a path connecting~$e$ to any target point $g^\bullet$. By compactness, we can find a finite number of points $g_\ell\in \gamma$, $\ell=0,1,\dots \bar \ell$, such that $g_0=e$, $g_{\bar \ell}=g^\bullet$ and $g_{\ell+1}\in g_\ell+\cU_e$ for every $\ell\in 0,\dots, (\bar \ell-1)$. Having proven that each~$g_{\ell+1}$ can be reached from~$g_\ell$ with a finite concatenation of loops, it follows that the system can be steered from~$e$ to~$g^\bullet$ in finitely many steps.

Gait controllability at $s^\star$ follows: thanks to our initial observation, each of these concatenated loops generates no overall shape displacement, and therefore the final shape will coincide with the initial one.

Condition~\eqref{abovecondition} holding at every $s^\star\in\cS$ immediately implies gait controllability according to Definition~\ref{PC} and this concludes the proof.
\end{proof}

The next corollary states that, if, in addition to the hypotheses of Theorem~\ref{th:gaitcontrollable}, a condition on the rank of the (constant) matrix~$F$ is assumed, then system~\eqref{system} is actually totally controllable.
\begin{corollary} \label{gait_tot}
	Let $\widetilde Z_i\in \R{n}\times \fg$ for $i=1,\dots, m$ be a family of equivariant $\cC^\infty$~vector fields on~$\cZ$. Assume that the matrix~$F$ is constant in~$s$ and has rank~$n=\dim(\cS)$.
	If~\eqref{abovecondition} holds true for some $s^\star\in \cS$, then system \eqref{system} is totally controllable.
\end{corollary}
\begin{proof}
By Theorem~\ref{th:gaitcontrollable}, the system is gait controllable at  $s^\star\in \cS$. 
Since $\rank F=n$, it follows by the classical Chow--Rashevskii Theorem~\ref{RC} that the shape-reduced system $\dot s= Fu$ is totally controllable. The thesis now follows from Proposition~\ref{PCimpliesTC}.
\end{proof}

We now present two examples illustrating how the converse implications to Proposition~\ref{TCtoPCtoFC} are, in general, false.

\begin{example}[gait controllability $\nRightarrow$ total controllability] \label{ex:notPCtoTC}
Let us consider a system \eqref{system} on an analytic manifold $\cZ$ with constant matrix $F$ with $\rank F=m< n$, and analytic $\xi_i$'s.  We also assume that \eqref{abovecondition} is satisfied for every $s^\star\in\cS$. On the one hand, by Theorem~\ref{th:gaitcontrollable}, the system is gait controllable. On the other hand, by Corollary~\ref{RCcor} we know that the system is totally controllable only if $\dim  \fz_{(s,e)^\top}=n+d$ for every $s\in \cS$. However, for every $s\in \cS$,
\begin{equation*}
	H_1(s)\cap\spann\{H_2(s),\dots,H_k(s),\dots\} =\{(0_{\R{m}},e_G)^\top\}
\end{equation*}
since the elements of the second set have a zero shape-component, while, by the linear independence of the $F_i$'s, the only element of $H_1(s)$ with a zero shape-component is $(0_{\R{m}},e_G)^\top$. Note now that $\dim H_1(s)=\rank F =m$ and $\dim(\spann\{H_2(s),\dots,H_k(s),\dots\})=d$ by \eqref{abovecondition}. Thus
\begin{equation*}
\dim  \fz_{(s,e)^\top} = \dim	H_1(s)+\dim (\spann\{H_2(s),\dots,H_k(s),\dots\}) =m+d<n+d\,,
\end{equation*}
whence the system is not totally controllable.

A concrete example illustrating this abstract situation is the following.
For $\cS=\R{3}$, $G=(\R{},+)$, and $\alpha\neq 0$, let us consider the following instance of \eqref{system} with:
\begin{align*}
F=	\begin{pmatrix}
		1& 0\\
		0&1\\
		\alpha& 0
	\end{pmatrix}\,, &&
\xi_1(s)=0\,, &&\xi_2(s)=s_1\,,
\end{align*} 
where the action of $G=(\R{},+)$ on itself is the translation, i.e., $gg'=g+g'$, whereas the action of $G$ on its Lie algebra is the identity, i.e., $g\xi=\xi$. Explicitly, \eqref{system} reads
$$\begin{pmatrix}
	\dot s_1\\
	\dot s_2\\
	\dot s_3\\
	\dot g
\end{pmatrix}=
\begin{pmatrix}
1\\0\\ \alpha\\0
\end{pmatrix}u_1+\begin{pmatrix}
0\\1\\0\\s_1
\end{pmatrix}u_2=\widetilde Z_1u_1+\widetilde Z_2u_2=(\Psi_g)_*(\widetilde Z_1u_1+\widetilde Z_2u_2).
$$
An easy computation delivers $[\widetilde Z_1,\widetilde Z_2]=(0,0,0,1)^\top$, thus, $\Pi^G_e[\widetilde Z_1,\widetilde Z_2]=1$ and condition \eqref{abovecondition} holds true.
\end{example}

\begin{remark} \label{rem:GCTC}
Intuitively, we expect the situation of Example~\ref{ex:notPCtoTC} to be relevant when the locomotor is provided with more independent controls than strictly necessary for locomotion, thus a larger dimension of the shape space; for instance, additional controlled joints might be added to increase efficiency or to accomplish other tasks. 
While for a fully functioning system we might expect $m\geq n$ (with the strict inequality corresponding to overactuation), mechanical failures or other causes might lead to the impossibility to use some of the controls, resulting in $m<n$. 
In such a case, total controllability in the shape might be lost, but gait controllability might still be available, which can be seen as a form of redundancy of the system for locomotion purposes. 
For example, in an $N$-link swimmer, gait controllability can be achieved by controlling only two angles, which alone do not provide total controllability, see, e.g., \cite{MMSZ}.
\end{remark}

\begin{remark}
Notice that, whereas for a constant matrix $F$ the condition $\rank F=m<n$ is not compatible with total controllability of the shape-reduced system, the property might still be achieved with a non-constant $F$, see, e.g.,~\cite{CM2011} for the case of an amoeba-like swimmer immersed in an ideal fluid.
\end{remark}

\begin{example}[fiber controllability $\nRightarrow$ gait  controllability] \label{ex:notFCtoPC}
Let us consider the system \eqref{system} with $\cS=\R{}$, $G=(\R{},+)$, $F(s)=1$ and $\xi(s)=\frac{1}{2}$. 

The system is fiber controllable. Indeed, given any initial configuration $z^\circ=(s^\circ,g^\circ)^\top$ and final position~$g^\bullet$, it is possible to reach $g^\bullet$ by applying the control $u(t)=\sign(g^\bullet-g^\circ)$ for a time $T=2\abs{g^\bullet-g^\circ}$. On the other hand, the system is not gait controllable, since any control producing a periodic shape change will result in a zero displacement in position. Indeed, we notice $s(0)=s(T)$ occurs if and only if the control has zero average on $[0,T]$. However, this implies that 
$$g(T)-g(0)=\int_0^T \dot g(t)\de t=\int_0^T \xi u(t)\de t=\int_0^T \frac{1}{2}u(t)\de t =0\,.
$$
\begin{figure}[tbh]
		\begin{center}
	\begin{tikzpicture}[line cap=round,line join=round,x=4mm,y=4mm, line width=1pt, scale=1.2, >=stealth]
		\clip(7,-0.5) rectangle (23,6);
		\draw [line width=1.2pt, fill=gray!40] (10,2.5)-- (10,0.5)-- (12,0.5)-- (12,2.5)-- (10,2.5);
		\draw (12,1.5)--(15.3,1.5);
		\draw (15.5,1.5)--(18,1.5);
		\draw (15.3,1.1)-- (14.5,1.1)--(14.5,1.9)--(15.3,1.9);
		\draw (14.7,1.3)-- (15.5,1.3)--(15.5,1.7)--(14.7,1.7);
		\draw [line width=1.2pt, fill=gray!40] (18,2.5)-- (18,0.5)-- (20,0.5)-- (20,2.5)-- (18,2.5);
		\draw [] (7,0.5)-- (23,0.5);	
		\fill [pattern = north east lines] (7,0.5) rectangle (23,-0.1);
		\draw[|-|] (11,3) -- (19,3);
		\draw (15,3.5) node {$s(t)$};
		\draw[<-] (19,3.5) -- (19,4.5) node[above] {$g(t)$};
	\end{tikzpicture}
\end{center}
	\caption{The incompetent crawler of Example~\ref{ex:notFCtoPC}.}\label{fig:2blocks}
\end{figure}
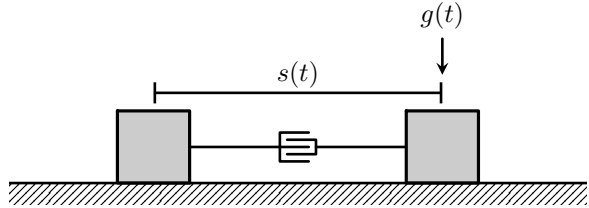

This system can be considered as a representation of the ``incompetent'' crawler in Figure~\ref{fig:2blocks}.
The shape describes the signed distance between the two blocks (to avoid overlappings when the sign changes, we can assume that the two block run on parallel tracks) and the position is given by the position of one of the two blocks. The equations of motion of the system describe the quasistatic locomotion of the crawler assuming (the same) viscous friction between each of the blocks and the substrate \cite{DeSTat}. In this case the crawler is incompetent, since it cannot properly locomote: indeed, it can be easily verified that the position of the midpoint between the two blocks is a constant of motion.
\end{example}

\begin{remark}\label{rem:GCFC}
	An additional advantage of gait controllability over fiber controllability can be observed looking at local versions of these two notions, namely, by considering, in Definitions~\ref{FC} and~\ref{PC}, only final positions $g^\bullet$ in a possibly small neighborhood $\cU_{g^\circ}$ of the initial position $g^\circ$.
	
	If the control vector fields are equivariant, then local gait controllability implies gait controllability. This follows from the fact that equivariance allows us to consider, without loss of generality, neighborhoods $\cU_g=g\cU_e$. Indeed, if from $e$ it is possible to reach every point in $\cU_e$, then from any $g$ it is possible to reach every point in $g\cU_{e}$; thus, any desired final position $g^\bullet$ can be reached by concatenating a finite number of steps.
	
	On the other hand, the same does not hold for fiber controllability. A counterexample can be easily produced by restricting the shape space in Example~\ref{ex:notFCtoPC} to a bounded open interval, e.g., $\cS=(-1,1)$. It is easily verified that the system is locally fiber controllable. However, since $\xi(s)=1/2$, from any starting position $g^\circ\in \R{}$ we can only reach positions not further away than half the diameter of $\cS$, as the midpoint of the blocks is stationary.
\end{remark}


\section{General framework: mechanics} \label{sec:generalmodel}

We start by recalling Resistive Force Theory~\cite{GH1955}, which is a well-known approximation of hydrodynamic forces and torques for one-dimensional bodies moving in a low Reynolds number flow in three dimensions.
In essence,  the linear density of hydrodynamic force~$\bef$ at a point~$\bx$ of the swimmer depends linearly on the velocity~$\bv$ at that point through two drag coefficients $C_\parallel\,,C_\perp>0$, measuring the resistance to motion in the parallel and perpendicular direction, respectively, to the unit tangent vector to the filament at~$\bx$.
In formulae, we have
\begin{subequations}\label{formula3.1}
\begin{equation}
\bef=C_\parallel (\bv\cdot\bt)\bt+C_\perp (\bv\cdot\bn)\bn\,,
\end{equation}
where~$\bt$ and~$\bn=\bt^{\perp}$ are the tangent and normal unit vectors, respectively, at~$\bx$.
Consequently, the linear density of hydrodynamic torque~$\bm$ is given by
\begin{equation}
\bm=(\bx-O)\times \bef\,,
\end{equation}
\end{subequations}
for a choice of the pole~$O$ with respect to which torques are computed.

It has been shown in~\cite{FRKHJ} that, for slender swimmers, a reasonable estimation for the ratio of the two drag coefficients is $C_\perp/C_\parallel=1.81\pm0.07$, so that, for the rest of the paper, we will consider
\begin{equation}\label{ratio}
\frac{C_\perp}{C_\parallel} \in(1,2],
\end{equation}
with $C_\perp/C_\parallel=2$ being the limit as the ratio of length and thickness of the swimmer tends to infinity~\cite{Lauga_book}.
We stress here that, in modeling uni-dimensional swimmers, the deformations that we will propose in the sequel involving stretching or elongation/compression of links will not affect the drag coefficients: in particular, we notice that possible changes in the length/thickness ratio have a vanishing effect on the ratio of the coefficients in the slender limit $C_\perp/C_\parallel\to 2$, so we will neglect them from our analysis.

\smallskip

\begin{figure}
	\begin{center}
\begin{tikzpicture}[line cap=round,line join=round,>=stealth,x=1cm,y=1cm, line width=1pt]
	\clip(-4.0,-2.) rectangle (5.,5.);
	\draw [shift={(0,0)},color=RoyalBlue,fill=RoyalBlue,fill opacity=0.1] (0,0) -- (0:2.0255) arc (0:69.0899:2.0255) -- cycle;	\draw [gray, dashed, line width=1pt] (0,0)-- (1.78,4.67);
	\draw [gray, dashed, line width=1pt] (-4,0)-- (5,0);
	\draw [shift={(0,0)},color=Orange,fill=Orange,fill opacity=0.1] (0,0) -- (36.8699:2.5319) arc (36.8699:69.0899:2.5319) -- cycle;
	\draw [shift={(0,0)},color=Orange,fill=Orange,fill opacity=0.1] (0,0) -- (69.0899:2.5319) arc (69.0899:101.3099:2.5319) -- cycle;
	\draw [line width=2pt] (0,0)-- (4,3);
	\draw [line width=2pt] (0,0)-- (-0.98,4.9);
	\draw [->,line width=2pt] (-3,-1.5)-- (-2,-1.5) node[anchor=west] {$\be_x$};
	\draw [->,line width=2pt] (-3,-1.5)-- (-3,-0.5) node[anchor=south] {$\be_y$};
	\draw [->,line width=1pt] (-3,-1.5)-- (0,0) node[anchor=north]{$\bh_t$};
	\draw [red,->,line width=2pt] (0,0) -- (0.8,0.6)node[anchor=west] {$\be^-_t$};
	\draw [red,->,line width=2pt] (0,0) -- (-0.196,0.98) node[anchor=east] {$\be^+_t$};
	\draw[Orange] (0.3,2.8) node {$\sigma_t$};
	\draw[Orange] (1.65,2.3) node {$\sigma_t$};
	\draw[RoyalBlue] (2.1,1) node {$\theta_t$};
\end{tikzpicture}
	\end{center}
\caption{The parametrization of the two-link swimmer given in Section~\ref{sec:generalmodel}.}

\label{fig:scallop}
\end{figure}
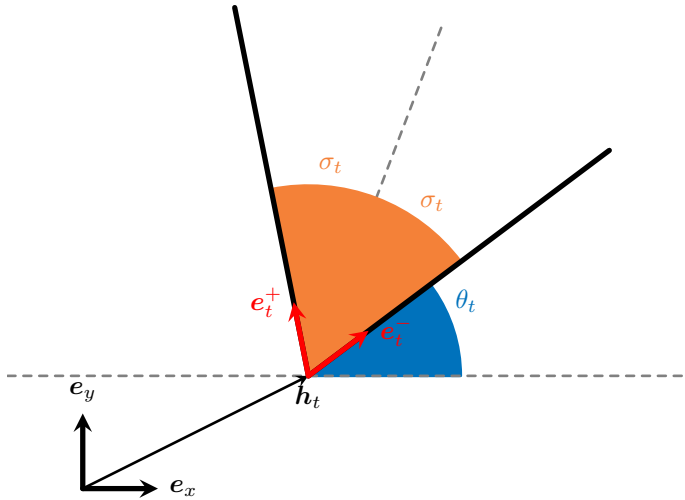	
To model a two-link swimmer that is free to move in the plane, we denote by $\bh_t=(x_t,y_t)\in\R{2}$ the position of the hinge with respect to the lab frame and by $\theta_t \in(-\pi,\pi)$ 
the angle that the bisector of the two links forms with the positive $x$-axis, cf.~Figure~\ref{fig:scallop}.
Let $2\sigma_t$ denote the opening angle between the links, which, in this paper, might change in time in a way directly depending on the control parameters. 
We will denote the two links with the symbols $+$ and $-$, and use $\pm$ collectively.

As argued in the Introduction, several possible mechanisms might be employed to obtain a change in the lengths of the links, each one yielding different velocity fields along them.
\begin{itemize}
	\item \emph{Active longitudinal strain}: the material can be actively deformed longitudinally along the links. That is, we are able to control the deformation gradient $\gamma_t^\pm$ of the material of the two links.  We use this mechanism in the \emph{stretching links} model of Section~\ref{sec:stretch} .
	\item \emph{Hinge movement}: the swimmer can be seen as a unique inextensible wire, bent at a point whose position along the wire is controlled. We describe mathematically this mechanisms as a change in the reference configurations $S_t^{\pm}$ of the two links on their hinge-side ends, with one elongating and the other shortening at the same rate. This case is studied in the \emph{sliding links} model of Section~\ref{sec:sliding}.
	\item \emph{Growth}: the length-change is produced by protruding and retracting body parts, that, consequently, are not always in contact with the fluid. 
	Also this case is described with a change in the reference configurations $S_t^{\pm}$ of the two links. 
	If such a change occurs at the ends of the links opposite to the hinge we have the \emph{growing links} model of Section~\ref{sec:growing}. 
	We also study a \emph{telescopic links} model in Section~\ref{sec:telescopic}, where the links elongate analogously to telescopic poles. More precisely, each link is divided into two sections, with the inner one $\widehat S_t^{\pm}$ sliding inside the outer one $S_t^{\pm}$, which corresponds to a change of the reference configurations $\widehat S_t^{\pm}$ of the inner sections at their junction with the outer ones. 
\end{itemize}

Denoting by $S_t^\pm\coloneqq[a_t^\pm,b_t^\pm]$ the reference configuration of the links of the two-link swimmer at time~$t$, the generic point in each link is described by
\begin{equation}\label{G01}
\bx_t^\pm(\eta)=\bh_t+\lambda_t^\pm(\eta)\be_t^\pm,\qquad \text{for $\eta\in S_t^\pm$,}
\end{equation}
where $\lambda_t^\pm\colon S_t^\pm\to[0,+\infty)$ identifies the position along the $\pm$-link at time~$t$, whereas the unit vectors $\be_t^\pm$ describe the directions of the two links. 
Recalling the properties of rotations in dimension two, they can be expressed as
\begin{equation*}
\be_t^\pm=\big(\cos(\theta_t\pm\sigma_t),\sin(\theta_t\pm\sigma_t)\big)\eqqcolon(\cos\alpha_t^\pm,\sin\alpha_t^\pm)\,,
\end{equation*}
or, equivalently, as
\begin{equation}\label{G03}
	\be_t^\pm=R_{\theta_t}(\cos\sigma_t,\pm\sin\sigma_t)=R_{\theta_t}R_{\pm\sigma_t}\be_1=R_{\theta_t\pm\sigma_t}\be_1=R_{\alpha_t^\pm}\be_1, 
\end{equation}
where $\be_1=(1,0)$ and $R_\alpha$ is the rotation of angle $\alpha$.

By requiring $\lambda_t^\pm (a_t^\pm)=0$,  we can conveniently express \eqref{G01} in terms of the deformation gradient $\gamma_t^\pm\colon S_t^\pm\to(0,+\infty)$ of the $\pm$-link as
\begin{equation}\label{G04}
\bx_t^\pm(\eta)=\bh_t+\bigg(\int_{a_t^\pm}^\eta \gamma_t^\pm(\varsigma)\,\de\varsigma\bigg)\be_t^\pm.
\end{equation}

In order to compute the hydrodynamic forces and torque acting on the swimmer, we need both the space derivative $(\bx_t^\pm)'(\eta)$ and the time derivative $\dot\bx_t^\pm(\eta)$:
\begin{equation}\label{G05}
\begin{split}
(\bx_t^\pm)'(\eta)=&\, \gamma_t^\pm(\eta)\be_t^\pm \,, \\
\bv_t(\bx_t^\pm(\eta))\coloneqq&\, \dot\bx_t^\pm(\eta)= \dot\bh_t+\bigg(\int_{a_t^\pm}^\eta \dot\gamma_t^\pm(\varsigma)\,\de\varsigma\bigg)\be_t^\pm-\gamma_t^\pm(a_t^\pm)\dot a_t^\pm \be_t^\pm+\bigg(\int_{a_t^\pm}^\eta \gamma_t^\pm(\varsigma)\,\de\varsigma\bigg)\dot\be_t^\pm \\
= &\, \dot\bh_t+\bigg(\int_{a_t^\pm}^\eta \dot\gamma_t^\pm(\varsigma)\,\de\varsigma\bigg)\be_t^\pm-\gamma_t^\pm(a_t^\pm)\dot a_t^\pm \be_t^\pm+\bigg(\int_{a_t^\pm}^\eta \gamma_t^\pm(\varsigma)\,\de\varsigma\bigg)\dot\alpha_t^\pm(\be_t^\pm)^\perp \\
= &\, R_{\theta_t}\bigg[R^{-1}_{\theta_t}\dot\bh_t+\bigg(\int_{a_t^\pm}^\eta \dot\gamma_t^\pm(\varsigma)\,\de\varsigma\bigg)R_{\pm\sigma_t}\be_1-\gamma_t^\pm(a_t^\pm)\dot a_t^\pm R_{\pm\sigma_t}\be_1 \\
&\, \phantom{R_{\theta_t}\bigg(}+\bigg(\int_{a_t^\pm}^\eta \gamma_t^\pm(\varsigma)\,\de\varsigma\bigg)\dot\alpha_t^\pm R_{\pm\sigma_t}\be_2\bigg]\,,
\end{split}
\end{equation}
where we have used \eqref{G03} to write $\bv(\bx_t)$ in equivalent ways.
Using Resistive Force Theory \cite{GH1955}, the line density of force and torque can be approximated by the parallel and orthogonal components of the space derivative $\dot{\bx}_t$ through suitable, body-dependent coefficients~$C_\parallel$ and~$C_\perp$, respectively. 
Therefore, the line densities of force~$\bef_t$ and torque~$\bm_t$ (the latter computed with respect to $\bh_t$) are 
\begin{equation*}
\begin{split}
\bef_t(\bv_t(\bx_t^\pm(\eta))) = &\, [(C_\parallel-C_\perp)\be_t^\pm\otimes\be_t^\pm+C_\perp I]\dot\bx_t^\pm(\eta) 
= R_{\alpha_t^\pm}[(C_\parallel-C_\perp)\bE_{11}+C_\perp I] R_{\alpha_t^\pm}^{-1} \dot\bx_t^\pm(\eta) \\
= &\, R_{\theta_t^\pm}R_{\pm\sigma_t} J R_{\pm\sigma_t}^{-1}R_{\theta_t}^{-1} \dot\bx_t^\pm(\eta), \\
\bm_t(\bv_t(\bx_t^\pm(\eta)),\bx_t^\pm(\eta))= &\,(\bx_t^\pm(\eta)-\bh_t)\times\bef_t(\bv_t(\bx_t^\pm(\eta))= \bigg(\int_{a_t^\pm}^\eta \gamma_t^\pm(\varsigma)\,\de\varsigma\bigg)\be_t^\pm \times \bef_t(\bv_t(\bx_t^\pm(\eta)),
\end{split}
\end{equation*}
where
$$\bE_{11}\coloneqq \be_1\otimes\be_1=\begin{pmatrix}1&0\\0&0\end{pmatrix}\qquad\text{and}\qquad J\coloneqq (C_\parallel-C_\perp)\bE_{11}+C_\perp I=\begin{pmatrix}C_\parallel&0\\0&C_\perp\end{pmatrix},$$
and $\bm_t(\bv_t(\bx_t^\pm(\eta)),\bx_t^\pm(\eta))=m_t(\bv_t(\bx_t^\pm(\eta)),\bx_t^\pm(\eta))\be_3\in\R3$.
Integrating over the links we obtain
\begin{subequations}\label{G06}
\begin{eqnarray}
\bF_t^\pm \!\!\!\!& = \!\!\!\!& \int_{\bx_t^\pm(S_t^\pm)}\bef_t(\bv_t(\xi))\,\de\xi=\, \int_{S_t^\pm}\bef_t(\bv_t(\bx_t^\pm(\eta)))\lVert(\bx_t^\pm)'(\eta)\rVert\,\de \eta = \int_{S_t^\pm}\bef_t(\dot\bx_t^\pm(\eta))\gamma_t^\pm(\eta)\,\de \eta, \label{G06a}\\ 
\bM_t^\pm \!\!\!\!& = \!\!\!\!& \int_{\bx_t^\pm(S_t^\pm)} \bm_t(\bv_t(\xi),\xi)\,\de \xi=\int_{S_t^\pm} m_t(\bv_t(\bx_t^\pm(\eta)),\bx_t^\pm(\eta))\lVert(\bx_t^\pm)'(\eta)\Vert\be_3\, \de \eta \nonumber\\
\!\!\!\!& = \!\!\!\!&\, \bigg( \int_{S_t^\pm} m_t(\bv_t(\bx_t^\pm(\eta)),\bx_t^\pm(\eta))\gamma_t^\pm(\eta)\de \eta\bigg)\be_3\eqqcolon M_t^\pm \be_3\,. \label{G06b}
\end{eqnarray}
\end{subequations}
The balance of total hydrodynamic forces and torques is obtained enforcing the self-propulsion constraint and reads
\begin{equation}\label{G07}
\begin{cases}\bF_t=\bF_t^++\bF_t^-=\bzero\,,\\
M_t= M_t^++M_t^- =0\,,
\end{cases}
\end{equation}
fully describing the evolution of the system.
A careful inspection of formulae~\eqref{G06} using~\eqref{formula3.1} shows that there is a linear dependence of the forces and torques on the velocities $\dot\bh_t$ and $\dot\theta_t$, and on the rate of deformation~$\dot \bs_t = (\dot \sigma_t,\dot a^+_t,\dot b^+_t,\dot a^-_t,\dot b^-_t,\dot \gamma^+_t, \dot \gamma^-_t)$ of the shape variables (which is the $s\in\cS$ in~\eqref{system}). 
Therefore, the structure of equation \eqref{G07} is
\begin{equation}\label{G07_2}
\bzero=\begin{pmatrix}
\bF_t\\
M_t\end{pmatrix}=\widetilde{\cR}_t\begin{pmatrix}\dot\bh_t\\\dot\theta_t\end{pmatrix}+\cV_t\dot\bs_t\,,
\end{equation}
where $\widetilde{\cR}_t$ is the so-called \emph{Grand Resistance Matrix} and $\cV_t$ is the matrix of the shape vector fields. 
In order to make the velocities $\dot\bh_t$ and $\dot\theta_t$ explicit in the equation of motion \eqref{G07_2}, we have to invert the matrix $\widetilde{\cR}_t$, which is possible since it is symmetric positive definite, see \cite[Section 5.2.1]{Lauga_book}.
In Section~\ref{sec:esempi} we discuss with more detail some specific instances of this general framework, but first we show how it can be reformulated in the control framework of Section~\ref{sec:control}.

\subsection{The mechanical problem as a control problem}
Since we work in a two-dimensional framework, the group $G$ describing the position of the swimmer is
\begin{equation*}
	G=\SE{2}\,,\qquad \text{with\qquad $G\ni g_t=(\bh_t,\theta_t)=(x_t,y_t,\theta_t)$,}
\end{equation*}
which in matrix notation can be written as
\begin{equation*}
	g_t=\begin{pmatrix}
		R_{\theta_t} & \bh_t\\
		(0,0) &1
	\end{pmatrix} \qquad\text{so that}\qquad \tilde gg=\begin{pmatrix}
	R_{\tilde \theta_t}R_{\theta_t} & \tilde \bh_t +R_{\tilde \theta_t}\bh_t \\
	(0,0) &1
\end{pmatrix}
\end{equation*}
The complete shape of the swimmer is given by the angle $\sigma_t\in(0,\pi)$ and either the functions $\lambda_t^\pm$ given in \eqref{G01}, or equivalently their deformation gradients $\gamma^\pm_t$; notice that the domains $S^\pm_t$ are part of the shape. In order to reduce the problem to finite dimension, we need to characterize the gradients $\gamma^\pm_t$ with only a finite set of parameters; in our models, we will simply assume that $\gamma^\pm_t$ is constant in $s$ along each of the arms.

Setting $D=\{(a,b)\colon b> a\}\subset \R{2}$, we have
\begin{equation}\label{eq:shape_space}
	\cS=(0,\pi)\times D^2\times \R{2}\,,\qquad \text{with $\cS\ni (\sigma_t,a^+_t,b^+_t,a^-_t,b^-_t,\gamma^+_t, \gamma^-_t)$}
\end{equation}
\begin{remark}
The shape space $\mathcal{S}$ in~\eqref{eq:shape_space} is the largest that we can consider. 
In concrete examples, we will restrict ourselves to smaller rectangular subsets (that is, still in the form of cartesian products of intervals) of $\mathcal{S}$ as shape space. 
On the one hand, this is not restrictive since the theorems used to prove controllability rely on Lie brackets computations which are linked only to small deformations; on the other hand, a bounded shape space is more realistic and allows us to identify different mechanisms of actuation and swimming strategies, both reducing the number of controls (and thus effective shape variables) and confining the possible shapes near a certain configuration (e.g., a scallop for $\sigma\approx \pi/4$ or a discretized filament for $\sigma\approx\pi/2$).
\end{remark}


\section{Three (plus one) models of swimmers} \label{sec:esempi}
In this section, we consider three (plus one) models of swimmers, differing in the way they deform.
To describe each of them, we will specify the sets $S_t^\pm$ and the functions $\gamma_t^\pm$ appropriately.

\subsection{Stretching links}\label{sec:stretch}
In this model, we consider a two-link swimmer whose links can elongate at a uniform (in~$\eta$) strain~$\gamma_t^\pm$ (notice that the strain function still depends on the arm), cf.~Figure~\ref{fig:stretching}.
We also set $S_t^\pm=[0,1]$.

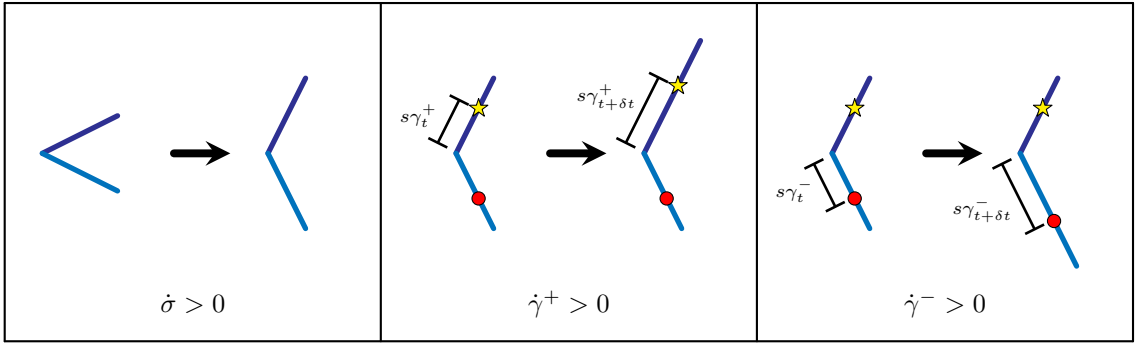
\begin{figure}[tb]
	\begin{center}
\begin{tikzpicture}[line cap=round,line join=round,>=stealth,x=5mm,y=5mm]
	\clip(-2,-5.5) rectangle (30,4.5);
	\draw [line width=2pt,color=Blue] (0,0)-- (2,1);
	\draw [line width=2pt,color=Blue] (6,0)-- (7,2);
	\draw [line width=2pt,color=Blue] (11,0)-- (12,2);
	\draw [line width=2pt,color=Blue] (16,0)-- (17.5,3);
	\draw [line width=2pt,color=Blue] (21,0)-- (22,2);
	\draw [line width=2pt,color=Blue] (26,0)-- (27,2);
	\draw [line width=2pt,color=RoyalBlue] (0,0)-- (2,-1);
	\draw [line width=2pt,color=RoyalBlue] (6,0)-- (7,-2);
	\draw [line width=2pt,color=RoyalBlue] (11,0)-- (12,-2);
	\draw [line width=2pt,color=RoyalBlue] (16,0)-- (17,-2);
	\draw [line width=2pt,color=RoyalBlue] (21,0)-- (22,-2);
	\draw [line width=2pt,color=RoyalBlue] (26,0)-- (27.5,-3);
	\draw [line width=0.8pt] (-1,4)-- (-1,-5);
	\draw [line width=0.8pt] (29,-5)-- (29,4);
	\draw [line width=0.8pt] (29,-5)-- (-1,-5);
	\draw [line width=0.8pt] (29,4)-- (-1,4);
	\draw [line width=0.8pt] (9,4)-- (9,-5);
	\draw [line width=0.8pt] (19,4)-- (19,-5);
	\draw [line width=1pt,|-|] (10.5,0.25)-- (11.1,1.45);
	\draw [line width=1pt,|-|] (15.5,0.25)-- (16.4,2.05);
	\draw [line width=1pt,|-|] (20.5,-0.25)-- (21.1,-1.45);
	\draw [line width=1pt,|-|] (25.5,-0.25)-- (26.4,-2.05);
	\draw [->,line width=3pt] (3.5,0) -- (5,0);
	\draw [->,line width=3pt] (13.5,0) -- (15,0);
	\draw [->,line width=3pt] (23.5,0) -- (25,0);
	\begin{scriptsize}
		\draw [] (11.6,1.2) node[star,star points=5, star point height=2.3pt, inner sep=1pt, color=black,fill=Yellow, draw] {};
		\draw [fill=red] (11.6,-1.2) circle (2.5pt);
		\draw [] (16.9,1.8) node[star,star points=5, star point height=2.3pt, inner sep=1pt, color=black,fill=Yellow, draw] {};
		\draw [fill=red] (16.6,-1.2) circle (2.5pt);
		\draw [] (21.6,1.2) node[star,star points=5, star point height=2.3pt, inner sep=1pt, color=black,fill=Yellow, draw] {};
		\draw [fill=red] (21.6,-1.2) circle (2.5pt);
		\draw [] (26.6,1.2) node[star,star points=5, star point height=2.3pt, inner sep=1pt, color=black,fill=Yellow, draw] {};
		\draw [fill=red] (26.9,-1.8) circle (2.5pt);
		\draw (10,1) node {$s\gamma^+_t $};
		\draw (15,1.5) node {$s\gamma^+_{t+\delta t} $};
		\draw (20,-1) node {$s\gamma^-_t $};
		\draw (25,-1.5) node {$s\gamma^-_{t+\delta t} $};
	\end{scriptsize}
	\draw (4,-4) node {$\dot \sigma>0$};
	\draw (14,-4) node {$\dot \gamma^+>0$};
	\draw (24,-4) node {$\dot \gamma^->0$};
\end{tikzpicture}
	\end{center}
	\caption{Illustration of the possible shape changes of the \emph{stretching links} model of Section~\ref{sec:stretch}. The red bullet and the yellow star represent the position of some exemplifying material points in consecutive steps.}
	\label{fig:stretching}
\end{figure}

Then \eqref{G04}, \eqref{G05}, and \eqref{G06} become
\begin{align*}
\bx_t^\pm(\eta)=&\, \bh_t+\eta\gamma_t^\pm\be_t^\pm\\
(\bx_t^\pm)'(\eta)=&\, \gamma^\pm_t\be_t^\pm \,, \\
\bv(\bx_t^\pm(\eta))=&\,
R_{\theta_t}\bigg(R^{-1}_{\theta_t}\dot\bh_t+\eta\dot\gamma_t^\pm R_{\pm\sigma_t}\be_1+\eta\gamma_t^\pm(\dot\theta_t\pm\dot\sigma_t) R_{\pm\sigma_t}\be_2\bigg) ,\\
 \bF_t^\pm=&\, \gamma_t^\pm R_{\theta_t} \bigg[R_{\pm\sigma_t}J R_{\pm\sigma_t}^{-1}R_{\theta_t}^{-1} \dot\bh_t +\frac{C_\parallel}2 \dot\gamma_t^\pm R_{\pm\sigma_t}\be_1+\frac{C_\perp}2 \gamma_t^\pm(\dot\theta_t\pm\dot\sigma_t) R_{\pm\sigma_t}\be_2 \bigg], \\
 \bM_t^\pm=&\, \bigg[\frac{C_\perp}2(\gamma_t^\pm)^2 (R_{\pm\sigma_t}\be_2){\cdot}(R_{\theta_t}^{-1}\dot\bh_t) +\frac{C_\perp}3(\gamma_t^\pm)^3(\dot\theta_t\pm\dot\sigma_t)\bigg]\be_3\,.
\end{align*}
The total balance of forces and torques acting on the swimmer \eqref{G07} become
\begin{equation}\label{207}
\begin{cases}
\bF_t=R_{\theta_t}\big[ A_t R_{\theta_t}^{-1}\dot\bh_t + \dot\theta_t\bb_t+\dot\sigma_t\tilde\bv^1_t+\dot\gamma_t^+\tilde\bv^{2,+}_t+\dot\gamma_t^-\tilde\bv^{2,-}_t\big]=\bzero\,,\\
 M_t=\bb_t\cdot(R_{\theta_t}^{-1}\dot\bh_t)+ c_t\dot\theta_t+v_t^3\dot\sigma_t=0\,,
\end{cases}
\end{equation}
where
\begin{equation*}
\begin{split}
A_t=&\, A(\sigma_t,\gamma_t^\pm;C_\parallel,C_\perp)\coloneqq \gamma_t^+ R_{\sigma_t}JR_{\sigma_t}^{-1}+\gamma_t^-R_{\sigma_t}^{-1}JR_{\sigma_t}\,, \\ 
\bb_t=&\, \bb(\sigma_t,\gamma_t^\pm;C_\perp)\coloneqq \frac{C_\perp}2\big((\gamma_t^+)^2R_{\sigma_t}+(\gamma_t^-)^2R_{\sigma_t}^{-1}\big)\be_2\,, \\
\tilde\bv^1_t=&\, \tilde\bv^1(\sigma_t,\gamma_t^\pm;C_\perp)\coloneqq \frac{C_\perp}2\big((\gamma_t^+)^2R_{\sigma_t}-(\gamma_t^-)^2R_{\sigma_t}^{-1}\big)\be_2\,, \\ 
\tilde\bv^{2,\pm}_t=&\, \tilde\bv^2(\sigma_t,\gamma_t^\pm;C_\parallel)\coloneqq \frac{C_\parallel}2 \gamma_t^\pm R_{\sigma_t}^{\pm1}\be_1\,,\\
c_t=&c(\gamma_t^\pm;C_\perp)\coloneqq \frac{C_\perp}3[(\gamma_t^+)^3+(\gamma_t^-)^3],\\
v_t^3=&v^3(\gamma_t^\pm;C_\perp)\coloneqq \frac{C_\perp}3[(\gamma_t^+)^3-(\gamma_t^-)^3]. 
\end{split}
\end{equation*}
The expressions in \eqref{207} can be gathered in matrix form as
\begin{equation*}
\begin{split}
\begin{pmatrix}\bF_t\\M_t\end{pmatrix}=\,&\begin{pmatrix}R_{\theta_t}&\bzero\\\bzero^\top&1\end{pmatrix}\left[\begin{pmatrix}A_t&\bb_t\\\bb_t^\top&c_t\end{pmatrix}\begin{pmatrix}R_{\theta_t}^{-1}\dot\bh_t\\\dot\theta_t\end{pmatrix}+\dot\sigma_t\begin{pmatrix}\tilde\bv_t^1\\ v_t^3\end{pmatrix}+\dot\gamma_t^+\begin{pmatrix}\tilde\bv_t^{2,+}\\0\end{pmatrix}+\dot\gamma_t^-\begin{pmatrix}\tilde\bv_t^{2,-}\\0\end{pmatrix}\right]\\
\eqqcolon& \begin{pmatrix}R_{\theta_t}&\bzero\\\bzero^\top&1\end{pmatrix}\left[\cR_t\begin{pmatrix}R_{\theta_t}^{-1}\dot\bh_t\\\dot\theta_t\end{pmatrix}+\dot\sigma_t\bv_t^1+\dot\gamma_t^+\bv_t^{2,+}+\dot\gamma_t^-\bv_t^{2,-}\right],
\end{split} 
\end{equation*}
where $\bv_t^1\coloneqq(\tilde\bv_t^1|v_t^3)$ and $\bv_t^{2,\pm}\coloneqq(\tilde\bv_t^{2,\pm}|0)$, and the $3\times3$ matrix $\cR_t=\cR(\sigma_t,\gamma_t^\pm;C_\parallel,C_\perp)$ is the Grand Resistance Matrix expressed in the swimmer's reference frame and, as we have seen before, it is invertible.
This yields
\begin{equation*}
\begin{pmatrix}R_{\theta_t}^{-1}\dot\bx_t\\\dot\theta_t\end{pmatrix}=-\cR_t^{-1}(\dot\sigma_t\bv_t^1+\dot\gamma_t^+\bv_t^{2,+}+\dot\gamma_t^-\bv_t^{2,-}),
\end{equation*}
which can be written in the form of a control system as
\begin{equation}\label{217_1}
\begin{pmatrix}R_{\theta_t}^{-1}\dot\bx_t\\ \dot\theta_t\\ \dot\sigma_t\\ \dot\gamma_t^+\\ \dot\gamma_t^-\end{pmatrix}=
\begin{pmatrix}-\cR_t^{-1}\bv_t^1\\ 1\\ 0\\ 0\end{pmatrix}u_1+
\begin{pmatrix}-\cR_t^{-1}\bv_t^{2,+}\\0\\1\\0\end{pmatrix}u_2+
\begin{pmatrix}-\cR_t^{-1}\bv_t^{2,-}\\0\\0\\1\end{pmatrix}u_3\eqqcolon\sum_{i=1}^3\bg_i(\sigma_t,\gamma_t^+,\gamma_t^-) u_i \,,
\end{equation}
where, for $i=1, 2, 3$, the maps $u_i\colon[0,T]\to\R{}$ are the control functions
identified with the time derivatives of the three shape parameters $\gamma_t^{\pm}$ and $\sigma_t$\,.
\begin{theorem}
In the regime~\eqref{ratio}, system \eqref{217_1} is totally controllable.
\end{theorem}
\begin{proof}

We compute the Lie brackets $\bg_t^4\coloneqq[\bg_t^1,\bg_t^2]$, $\bg_t^5\coloneqq[\bg_t^1,\bg_t^3]$, and $\bg_t^6\coloneqq[\bg_t^2,\bg_t^3]$ at the specific values $\sigma_\star=\pi/4$, $\gamma^{+}_{\star}=\gamma^{-}_{\star}=\ell>0$ obtaining the fields $\bg^i_\star\coloneqq \Pi^G_e(\bg^i(\sigma_\star,\gamma^\pm_\star))$ (see \eqref{Pi_g}), for $i=4,5,6$,
\begin{equation}\label{218}
\begin{split}
\!\!\!\! \bg^4_\star= &\, \bigg(
\displaystyle \frac{-(C_\perp-C_\parallel)^2}{2\sqrt2(C_\parallel+C_\perp)^2},
\displaystyle \frac{16C_\parallel^3-8C_\parallel^2C_\perp-23C_\parallel C_\perp^2-2C_\perp^3}{2\sqrt2(C_\parallel+C_\perp)(C_\perp+4C_\parallel)^2},
\displaystyle \frac{3(8C_\parallel^3+10C_\parallel^2C_\perp+9C_\parallel C_\perp^2+C_\perp^3)}{2\ell(C_\parallel+C_\perp)(C_\perp+4C_\parallel)^2}
\bigg)^\top,\\
\!\!\!\! \bg^5_\star= &\, \bigg(
\displaystyle \frac{-(C_\perp-C_\parallel)^2}{2\sqrt2(C_\parallel+C_\perp)^2},
\displaystyle -\frac{16C_\parallel^3-8C_\parallel^2C_\perp-23C_\parallel C_\perp^2-2C_\perp^3}{2\sqrt2(C_\parallel+C_\perp)(C_\perp+4C_\parallel)^2},
\displaystyle -\frac{3(8C_\parallel^3+10C_\parallel^2C_\perp+9C_\parallel C_\perp^2+C_\perp^3)}{2\ell(C_\parallel+C_\perp)(C_\perp+4C_\parallel)^2}
\bigg)^\top,\\
\!\!\!\! \bg^6_\star= &\, \bigg(
0,
\displaystyle -\frac{C_\parallel(2C_\perp+3C_\parallel)}{2\sqrt2\ell(C_\parallel+C_\perp)(C_\perp+4C_\parallel)},
\displaystyle \frac{3C_\parallel(C_\parallel-C_\perp)}{2\ell^2(C_\parallel+C_\perp)(C_\perp+4C_\parallel)}
\bigg)^\top.
\end{split}
\end{equation}
To apply Corollary \ref{gait_tot}, we have to verify that the vector fields~$\bg^4_\star$,~$\bg^5_\star$, and~$\bg^6_\star$ are linearly independent so that they generate the Lie algebra $\fg$. 
The linear independence of these vector fields follows from the fact that 
$$\det\big(\bg_\star^4|\bg_\star^5|\bg_\star^6\big)=\frac{3C_\parallel^3(C_\perp-C_\parallel)^2(31C_\perp+4C_\parallel)}{4\ell^2(C_\perp+C_\parallel)^3(C_\perp+4C_\parallel)^3}\neq0
$$
since, owing to~\eqref{ratio}, $C_\perp\neq C_\parallel$.
Therefore, condition \eqref{abovecondition} is fulfilled and the systems is gait controllable by Theorem~\ref{th:gaitcontrollable}.
Moreover, we notice from \eqref{217_1} that the matrix $F=\mathbb{I}_3$  defining the equations of the shape variables is constant and with rank $3$. 
The total controllability of the system follows now from Corollary~\ref{gait_tot} and this concludes the proof.
\end{proof}

\subsection{Sliding links} \label{sec:sliding}
In this section, we consider a two-link swimmer whose body consists of an unstretchable wire of constant length $L>0$, bent at a hinge point $\bh_t$ in such a way that the two portions of the wire form an angle $2\sigma_t$ between them, with the relative position of the hinge with respect to the wire that can vary in time,  cf.~Figure~\ref{fig:sliding}. In this case we set $S^+_t\coloneqq[a_t,0]$ and $S^-_t\coloneqq[-L-a_t,0]$ for a  function $t\mapsto a_t\in (-L+\beta,-\beta)$ for a certain $\beta>0$ small enough (e.g., $\beta\in(0,L/4)$); moreover we let  $\gamma^\pm_t\equiv 1$.

\begin{figure}[tb]
	\begin{center}
	\begin{tikzpicture}[line join=round,>=stealth,x=5mm,y=5mm]
	\clip(-2,-6) rectangle (30,4.5);
	\draw [line width=2pt,color=Blue] (0,0)-- (1.2,2.4);
	\draw [line width=2pt,color=Blue] (5,-0.8)-- (6.6,2.4);
	\draw [line width=2pt,color=Blue] (10,-1.6)-- (12,2.4);
	\draw [line width=2pt,color=Blue] (0,0)-- (1.2,-2.4);
	\draw [line width=2pt,color=Blue] (5,-0.8)-- (5.8,-2.4);
	\draw [line width=2pt,color=Blue] (10,-1.6)-- (10.4,-2.4);
	\draw [line width=2pt,color=Blue] (15,0)-- (16.2,2.4);
	\draw [line width=2pt,color=Blue] (20,0.8)-- (21.6,-2.4);
	\draw [line width=2pt,color=Blue] (25,1.6)-- (27,-2.4);
	\draw [line width=2pt,color=Blue] (15,0)-- (16.2,-2.4);
	\draw [line width=2pt,color=Blue] (20,0.8)-- (20.8,2.4);
	\draw [line width=2pt,color=Blue] (25,1.6)-- (25.4,2.4);
	\draw [line width=0.8pt] (-1,3.5)-- (-1,-5);
	\draw [line width=0.8pt] (29,-5)-- (29,3.5);
	\draw [line width=0.8pt] (29,-5)-- (-1,-5);
	\draw [line width=0.8pt] (29,3.5)-- (-1,3.5);
	\draw [line width=0.8pt] (14,3.4)-- (14,-5);
	\draw [->,line width=3pt] (2.,0) -- (3.5,0);
	\draw [->,line width=3pt] (7.5,0) -- (9,0);	
	\draw [->,line width=3pt] (17.,0) -- (18.5,0);
	\draw [->,line width=3pt] (22.5,0) -- (24,0);
	\begin{scriptsize}
		\draw [fill=red] (0,0) circle (2.5pt);
		\draw [fill=red] (5.4,0) circle (2.5pt);
		\draw [fill=red] (10.8,0) circle (2.5pt);
		\draw  (0.4,0.8) node[diamond, inner sep=1.5pt, color=black,fill=white,draw] {};
		\draw  (5.8,0.8) node[diamond, inner sep=1.5pt, color=black,fill=white,draw] {};
		\draw  (11.2,0.8) node[diamond, inner sep=1.5pt, color=black,fill=white,draw] {};
		\draw  (0.4,-0.8) node[star,star points=5, star point height=2.3pt, inner sep=1pt, color=black,fill=Yellow, draw] {};
		\draw  (5,-0.8) node[star,star points=5, star point height=2.3pt, inner sep=1pt, color=black,fill=Yellow, draw] {};
		\draw  (10.4,-0.8) node[star,star points=5, star point height=2.3pt, inner sep=1pt, color=black,fill=Yellow, draw] {};
		\draw [fill=red] (15,0) circle (2.5pt);
		\draw [fill=red] (20.4,0) circle (2.5pt);
		\draw [fill=red] (25.8,0) circle (2.5pt);
		\draw  (15.4,0.8) node[diamond, inner sep=1.5pt, color=black,fill=white,draw] {};
		\draw  (20,0.8) node[diamond, inner sep=1.5pt, color=black,fill=white,draw] {};
		\draw  (25.4,0.8) node[diamond, inner sep=1.5pt, color=black,fill=white,draw] {};
		\draw  (15.4,-0.8) node[star,star points=5, star point height=2.3pt, inner sep=1pt, color=black,fill=Yellow, draw] {};
		\draw  (20.8,-0.8) node[star,star points=5, star point height=2.3pt, inner sep=1pt, color=black,fill=Yellow, draw] {};
		\draw  (26.2,-0.8) node[star,star points=5, star point height=2.3pt, inner sep=1pt, color=black,fill=Yellow, draw] {};
	\end{scriptsize}
	\draw (6.5,-4) node {\emph{Sliding links}: $\dot{a}_t^+<0$};
	\draw (21.5,-4) node {\emph{Sliding links}: $\dot{a}_t^+>0$};
\end{tikzpicture}
	\end{center}
	\caption{Illustration of the elongation-type shape changes of the \emph{sliding links} model of Section~\ref{sec:sliding}. The red bullet, white square and yellow star represent the position of some exemplifying material points in consecutive steps.} \label{fig:sliding}
\end{figure}
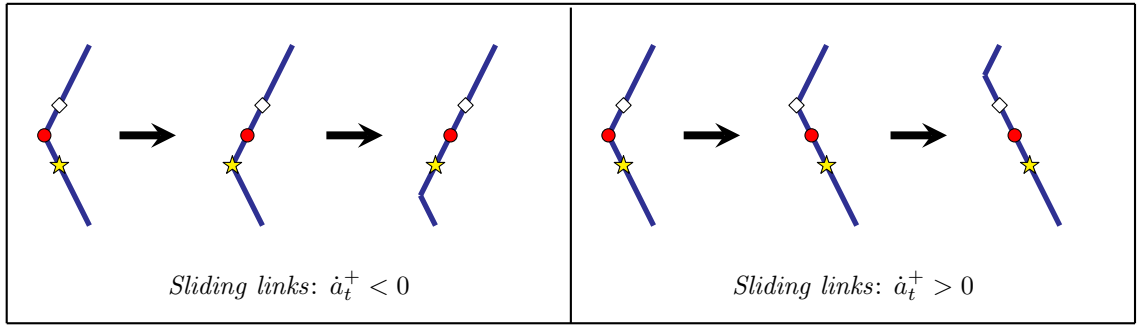

Therefore \eqref{G04}, \eqref{G05}, and \eqref{G06} become
\begin{align*}
\bx_t^+(\eta)=&\bh_t+(\eta-a_t)\be_t^+\,,\qquad\bx_t^-(\eta)=\bh_t+(\eta+L+a_t)\be_t^-\,,\\
(\bx_t^\pm)'(\eta)=&\, \be_t^\pm\,, \\
\bv(\bx_t^+(\eta))=&\, 
R_{\theta_t}\bigg(R^{-1}_{\theta_t}\dot\bh_t-\dot a_tR_{\sigma_t}\be_1+(\eta-a_t)(\dot\theta_t+\dot\sigma_t) R_{\sigma_t}\be_2\bigg),\\
\bv(\bx_t^-(\eta))=&\, R_{\theta_t}\bigg(R^{-1}_{\theta_t}\dot\bh_t+\dot a_tR_{-\sigma_t}\be_1+(\eta+L+a_t)(\dot\theta_t-\dot\sigma_t) R_{-\sigma_t}\be_2\bigg),\\
 \bF_t^+=&\, a_t R_{\theta_t} \bigg[R_{\sigma_t}J R_{\sigma_t}^{-1}R_{\theta_t}^{-1} \dot\bh_t -C_\parallel \dot a_tR_{\sigma_t}\be_1-\frac{C_\perp}2 a_t(\dot\theta_t+\dot\sigma_t) R_{\sigma_t}\be_2 \bigg], \\
  \bF_t^-=&\, (L+a_t) R_{\theta_t} \bigg[R_{-\sigma_t}J R_{-\sigma_t}^{-1}R_{\theta_t}^{-1} \dot\bh_t +C_\parallel \dot a_tR_{-\sigma_t}\be_1+\frac{C_\perp}2 (L+a_t)(\dot\theta_t-\dot\sigma_t) R_{-\sigma_t}\be_2 \bigg], \\
\bM_t^+=& \bigg[-\frac{C_\perp}2 a_t^2 (R_{\sigma_t}\be_2){\cdot}(R_{\theta_t}^{-1}\dot\bh_t) +\frac{C_\perp}3 a_t^3(\dot\theta_t+\dot\sigma_t)\bigg]\be_3\,,
\\
\bM_t^-=& \bigg[\frac{C_\perp}2(L+a_t)^2 (R_{-\sigma_t}\be_2){\cdot}(R_{\theta_t}^{-1}\dot\bh_t) +\frac{C_\perp}3(L+a_t)^3(\dot\theta_t-\dot\sigma_t)\bigg]\be_3\,.
\end{align*}

The total balance of forces and torques acting on the swimmer \eqref{G07} becomes
\begin{equation}\label{207_2}
\begin{cases}
\bF_t=R_{\theta_t}\big[ A_t R_{\theta_t}^{-1}\dot\bh_t + \dot\theta_t\bb_t+\dot\sigma_t\tilde\bv^1_t+\dot a_t\tilde\bv^2_t\big]=\bzero\,,\\
 M_t=\bb_t\cdot(R_{\theta_t}^{-1}\dot\bh_t)+ c_t\dot\theta_t+w_t\dot\sigma_t=0\,,
\end{cases}
\end{equation}
where
\begin{equation*}
\begin{split}
A_t=&\, A(\sigma_t,a_t;C_\parallel,C_\perp)\coloneqq a_t R_{\sigma_t}JR_{\sigma_t}^{-1}+(L+a_t)R_{\sigma_t}^{-1}JR_{\sigma_t}\,, \\ 
\bb_t=&\, \bb(\sigma_t,a_t;C_\perp)\coloneqq \frac{C_\perp}2\big((L+a_t)^2R_{\sigma_t}^{-1}-(a_t)^2R_{\sigma_t}\big)\be_2\,, \\
\tilde\bv^1_t=&\, \tilde\bv^1(\sigma_t,a_t;C_\perp)\coloneqq -\frac{C_\perp}2\big((a_t)^2R_{\sigma_t}+(L+a_t)^2R_{\sigma_t}^{-1}\big)\be_2\,, \\ 
\tilde\bv^{2}_t=&\, \tilde\bv^2(\sigma_t,a_t;C_\parallel)\coloneqq C_\parallel\big((L+a_t) R_{\sigma_t}^{-1}-a_t R_{\sigma_t}\big)\be_1\,,\\
c_t=&c(a_t;C_\perp)\coloneqq \frac{C_\perp}3[(a_t)^3+(L+a_t)^3],\\
w_t=&w(a_t;C_\perp)\coloneqq \frac{C_\perp}3[(a_t)^3-(L+a_t)^3]. 
\end{split}
\end{equation*}
The expressions in \eqref{207_2} can be gathered in matrix form as
\begin{equation*}
\begin{split}
\begin{pmatrix}\bF_t\\M_t\end{pmatrix}=\,&\begin{pmatrix}R_{\theta_t}&\bzero\\\bzero^\top&1\end{pmatrix}\left[\begin{pmatrix}A_t&\bb_t\\\bb_t^\top&c_t\end{pmatrix}\begin{pmatrix}R_{\theta_t}^{-1}\dot\bh_t\\\dot\theta_t\end{pmatrix}+\dot\sigma_t\begin{pmatrix}\tilde\bv_t^1\\ w_t\end{pmatrix}+\dot a_t\begin{pmatrix}\tilde\bv_t^{2}\\0\end{pmatrix}\right]\\
\eqqcolon& \begin{pmatrix}R_{\theta_t}&\bzero\\\bzero^\top&1\end{pmatrix}\left[\cR_t\begin{pmatrix}R_{\theta_t}^{-1}\dot\bh_t\\\dot\theta_t\end{pmatrix}+\dot\sigma_t\bv_t^1+\dot a_t\bv_t^{2}\right],
\end{split} 
\end{equation*}
where, as in the previous section, $\bv_t^1\coloneqq(\tilde\bv_t^1|w_t)$ and $\bv_t^{2}\coloneqq(\tilde\bv_t^{2}|0)$, and the $3\times3$ matrix $\cR_t=\cR(\sigma_t, a_t;C_\parallel,C_\perp)$ is the Grand Resistance Matrix expressed in the swimmer's reference frame and, as we have seen before, it is invertible. This yields
\begin{equation*}
\begin{pmatrix}R_{\theta_t}^{-1}\dot\bx_t\\\dot\theta_t\end{pmatrix}=-\cR_t^{-1}(\dot\sigma_t\bv_t^1+\dot a_t\bv_t^{2}),
\end{equation*}
which can be written in the form of a control system as
\begin{equation}\label{217_2}
\begin{pmatrix}R_{\theta_t}^{-1}\dot\bx_t\\ \dot\theta_t\\ \dot\sigma_t\\ \dot a_t\end{pmatrix}=
\begin{pmatrix}-\cR_t^{-1}\bv_t^1\\ 1\\ 0\end{pmatrix}u_1+
\begin{pmatrix}-\cR_t^{-1}\bv_t^{2}\\0\\1\end{pmatrix}u_2
\eqqcolon\sum_{i=1}^2\bg_i(\sigma_t, a_t) u_i \,,
\end{equation}
where, for $i=1, 2$, the maps $u_i\colon[0,T]\to\R{}$ are the control functions identified with the time derivatives of the two shape parameters $\sigma_t$ and $a_t$\,. 
\begin{theorem}
In the regime~\eqref{ratio}, system \eqref{217_2} is totally controllable.
\end{theorem}
\begin{proof}
In order to apply Corollay \ref{gait_tot} we compute the Lie brackets $\bg_t^3\coloneqq[\bg_t^1,\bg_t^2]$, $\bg_t^4\coloneqq[\bg_t^1,\bg_t^3]$ and $\bg_t^5\coloneqq[\bg_t^2,\bg_t^3$], at the specific values $\sigma_\star=\pi/4$, $a_\star=-\frac{L}2$, obtaining the fields $\bg_\star^i\coloneqq \Pi^G_e(\bg_i(\sigma_\star,a_\star))$ (see \eqref{Pi_g}), for $i=3,4,5$,
\begin{equation}\label{218_slide}
\begin{split}
\bg^3_\star= &\, \begin{pmatrix}0, &
\displaystyle \frac{\sqrt{2}(C_\perp^3+19 C_\perp^2C_\parallel+8C_\perp C_\parallel^2-16C_\parallel^3)}{(C_\parallel+C_\perp)(C_\perp+4C_\parallel)^2}, &
\displaystyle- \frac{6C_\perp(C_\perp^2+11C_\perp C_\parallel+4C_\parallel^2)}{L(C_\parallel+C_\perp)(C_\perp+4C_\parallel)^2}\end{pmatrix}^\top\\
\bg^4_\star= &\, \begin{pmatrix}
0\\[2mm]
\displaystyle \frac{5C_\perp^5+117C_\perp^4C_\parallel-466C_\perp^3C_\parallel^2-1128C_\perp^2C_\parallel^3+32C_\perp C_\parallel^4+384C_\parallel^5}{\sqrt2(C_\parallel+C_\perp)^2(C_\perp+4C_\parallel)^3}\\[4mm]
\displaystyle \frac{24C_\parallel(C_\perp^4+65C_\perp^3C_\parallel+110C_\perp^2C_\parallel^2+32C_\perp C\parallel^3-32C_\parallel^4}{L(C_\parallel+C_\perp)^2(C_\perp+4C_\parallel)^3}
\end{pmatrix},\\
\bg^5_\star= &\, \begin{pmatrix}
\displaystyle \frac{(2\sqrt2(C_\perp-C_\parallel)(3C_\perp^3+23C_\perp^2C_\parallel+74C_\perp C_\parallel^2+72C_\parallel^3}{L(C_\parallel+C_\perp)^2(C_\perp+4C_\parallel)^2}, & 0, &0 \end{pmatrix}^\top
\end{split}
\end{equation}
To apply Corollary \ref{gait_tot}, we have to verify that the vector fields $\bg^3_\star$\,, $\bg^4_\star$, and $\bg^5_\star$ are linearly independent so that they generate the Lie algebra $\fg$.
An explicit computation yields
$$
\det(\bg_\star^3|\bg_\star^4|\bg_\star^5)=p(C_\perp,C_\parallel)q(C_\perp,C_\parallel),
$$
where
$$
p(C_\perp,C_\parallel)\coloneqq\frac{12C_\parallel^5(C_\perp-C_\parallel)^2(3C_\perp^3+23C_\perp^2 C_\parallel+74C_\perp C_\parallel^2+72C_\parallel^3)}{L^2(C_\parallel+C_\perp)^5(C_\perp+4C_\parallel)^5},
$$
and, denoting $\kappa\coloneqq C_\perp/C_\parallel$\,,
$$
q(\kappa)\coloneqq 5\kappa^5+145\kappa^4+458\kappa^3+752\kappa^2+160\kappa-256.
$$
In the regime~\eqref{ratio}, we have that  $p(C_\perp,C_\parallel)>0$ and $q(\kappa)>0$. Thus, the determinant is always different from zero and the linear independence of $\bg^3_\star$\,, $\bg^4_\star$, and $\bg^5_\star$ follows.
Therefore, condition \eqref{abovecondition} is fulfilled and the systems is gait controllable by Theorem~\ref{th:gaitcontrollable}.
Moreover note that from  \eqref{217_2} the matrix $F$ defining the equations of the shape variables is constant and with rank $2$. Therefore we can apply Corollary \ref{gait_tot} to conclude that the system is totally controllable.
\end{proof}

\subsection{Growing links}\label{sec:growing}

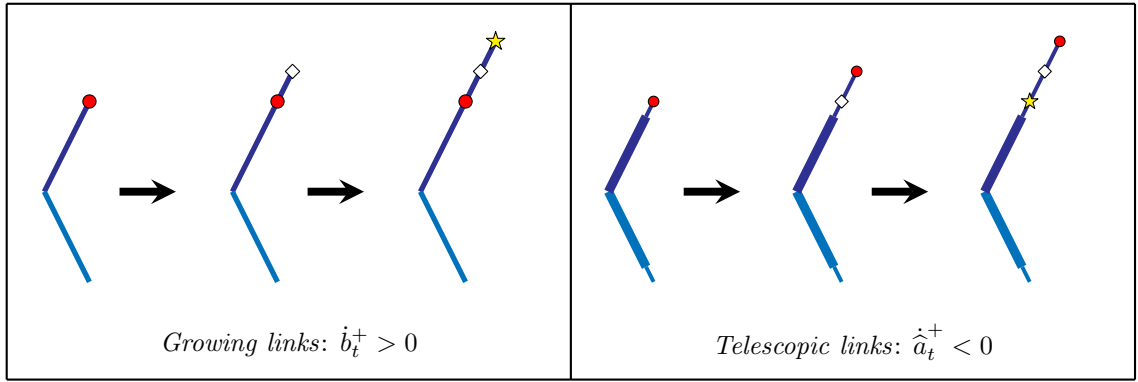
\begin{figure}[tb]
	\begin{center}
	\begin{tikzpicture}[line join=round,>=stealth,x=5mm,y=5mm]
	\clip(-2,-6) rectangle (30,6);
	\draw [line width=2pt,color=Blue] (0,0)-- (1.2,2.4);
	\draw [line width=2pt,color=Blue] (5,0)-- (6.6,3.2);
	\draw [line width=2pt,color=Blue] (10,0)-- (12,4);
	\draw [line width=2pt,color=RoyalBlue] (0,0)-- (1.2,-2.4);
	\draw [line width=2pt,color=RoyalBlue] (5,0)-- (6.2,-2.4);
	\draw [line width=2pt,color=RoyalBlue] (10,0)-- (11.2,-2.4);
	\draw [line width=0.8pt] (-1,5)-- (-1,-5);
	\draw [line width=0.8pt] (29,-5)-- (29,5);
	\draw [line width=0.8pt] (29,-5)-- (-1,-5);
	\draw [line width=0.8pt] (29,5)-- (-1,5);
	\draw [line width=0.8pt] (14,5)-- (14,-5);
	\draw [->,line width=3pt] (2.,0) -- (3.5,0);
	\draw [->,line width=3pt] (7.,0) -- (8.5,0);	
	\draw [->,line width=3pt] (17.,0) -- (18.5,0);
	\draw [->,line width=3pt] (22.,0) -- (23.5,0);
	\draw [line width=3.5pt,color=RoyalBlue] (15,0)-- (16,-2);
	\draw [line width=3.5pt,color=RoyalBlue] (20,0)-- (21,-2);
	\draw [line width=3.5pt,color=RoyalBlue] (25,0)-- (26,-2);
	\draw [line width=1.5pt,color=RoyalBlue] (16,-2) -- (16.2,-2.4);
	\draw [line width=1.5pt,color=RoyalBlue] (21,-2)-- (21.2,-2.4);
	\draw [line width=1.5pt,color=RoyalBlue] (26,-2)-- (26.2,-2.4);
	\draw [line width=3.5pt,color=Blue] (15,0)-- (16,2);
	\draw [line width=1.5pt,color=Blue] (16,2)-- (16.2,2.4);
	\draw [line width=3.5pt,color=Blue] (20,0)-- (21,2);
	\draw [line width=1.5pt,color=Blue] (21,2)-- (21.6,3.2);
	\draw [line width=3.5pt,color=Blue] (25,0)-- (26,2);
	\draw [line width=1.5pt,color=Blue] (26,2)-- (27,4);
	\begin{scriptsize}
		\draw [fill=red] (1.2,2.4) circle (2.5pt);
		\draw [fill=red] (6.2,2.4) circle (2.5pt);
		\draw  (6.6,3.2) node[diamond, inner sep=1.4pt, color=black,fill=white, draw] {};
		\draw [fill=red] (11.2,2.4) circle (2.5pt);
		\draw  (11.6,3.2) node[diamond, inner sep=1.4pt, color=black,fill=white, draw] {};
		\draw [] (12,4) node[star,star points=5, inner sep=1pt,  star point height=2.3pt, color=black,fill=Yellow, draw] {};
		\draw [fill=red] (16.2,2.4) circle (2.pt);
		\draw [fill=white] (21.2,2.4)node[diamond, inner sep=1.3pt, color=black,fill=white, draw] {};
		\draw [fill=red] (21.6,3.2) circle (2.pt);
		\draw  (26.2,2.4) node[star,star points=5,  star point height=2pt, inner sep=0.9pt, color=black,fill=Yellow, draw] {};
		\draw (26.6,3.2) node[diamond, inner sep=1.3pt, color=black,fill=white, draw] {};
		\draw [fill=red] (27,4) circle (2.pt);
	\end{scriptsize}
	\draw (6.5,-4) node {\emph{Growing links}: $\dot b_t^+>0$};
	\draw (21.5,-4) node {\emph{Telescopic links}: $\dot{\widehat{a}}_t^+<0$};
\end{tikzpicture}
	\end{center}
	\caption{Illustration of the elongation of the $+$-link in the \emph{growing links} model of Section~\ref{sec:growing} (left) and in the \emph{telescopic links} model of Section~\ref{sec:telescopic} (right). The red bullet, white square and yellow star represent the position of some exemplifying material points in consecutive steps. The appearance of additional material points in successive steps is due to the eversion/protusion mechanism.} \label{fig:grow_tele}
\end{figure}

In this section we consider a two-link swimmer capable of changing the lengths of the links by growing at the two tips opposite to the hinge,  cf.~Figure~\ref{fig:grow_tele}.
In this case we set $S^+_t\coloneqq[0,b_t^+]$ and $S^-_t\coloneqq[0,b^-_t]$ with $b_t^\pm\in(0,L^\mathrm{max})$; moreover we let  $\gamma^\pm_t\equiv 1$. Therefore \eqref{G04}, \eqref{G05}, and \eqref{G06} become
\begin{align*}
		\bx_t^+(\eta)=&\bh_t+\eta\be_t^+\,, \qquad\bx_t^-(\eta)=\bh_t+\eta\be_t^-\,,\\
		(\bx_t^\pm)'(\eta)=&\, \be_t^\pm \,, \\
		\bv(\bx_t^+(\eta))=&\,  
R_{\theta_t}\bigg(R^{-1}_{\theta_t}\dot\bh_t+\eta(\dot\theta_t+\dot\sigma_t) R_{\sigma_t}\be_2\bigg) ,\\
\bv(\bx_t^-(\eta))=&\, R_{\theta_t}\bigg(R^{-1}_{\theta_t}\dot\bh_t+\eta(\dot\theta_t-\dot\sigma_t) R_{-\sigma_t}\be_2\bigg),\\
 \bF_t^+=&\, b_t^+ R_{\theta_t} \bigg[R_{\sigma_t}J R_{\sigma_t}^{-1}R_{\theta_t}^{-1} \dot\bh_t-\frac{C_\perp}2 b_t^+(\dot\theta_t+\dot\sigma_t) R_{\sigma_t}\be_2 \bigg], \\
  \bF_t^-=&\, b_t^- R_{\theta_t} \bigg[R_{-\sigma_t}J R_{-\sigma_t}^{-1}R_{\theta_t}^{-1} \dot\bh_t +\frac{C_\perp}2 b_t^-(\dot\theta_t-\dot\sigma_t) R_{-\sigma_t}\be_2 \bigg], \\
  \bM_t^+=& \bigg[-\frac{C_\perp}2 (b_t^+)^2 (R_{\sigma_t}\be_2){\cdot}(R_{\theta_t}^{-1}\dot\bh_t) +\frac{C_\perp}3 (b_t^+)^3(\dot\theta_t+\dot\sigma_t)\bigg]\be_3\,,\\
\bM_t^-=& \bigg[\frac{C_\perp}2(b_t^-)^2 (R_{-\sigma_t}\be_2){\cdot}(R_{\theta_t}^{-1}\dot\bh_t) +\frac{C_\perp}3(b_t^-)^3(\dot\theta_t-\dot\sigma_t)\bigg]\be_3\,.
\end{align*}
The total balance of forces and torques acting on the swimmer \eqref{G07} become
\begin{equation*}
\begin{cases}
\bF_t=R_{\theta_t}\big[ A_t R_{\theta_t}^{-1}\dot\bh_t + \dot\theta_t\bb_t+\dot\sigma_t\tilde\bv^1_t+\dot b_t^+\tilde\bv^2_t+\dot b_t^-\tilde\bv^3_t\big]=\bzero\,,\\
 M_t=\bb_t\cdot(R_{\theta_t}^{-1}\dot\bh_t)+ c_t\dot\theta_t+w_t\dot\sigma_t=0\,,
\end{cases}
\end{equation*}
where
\begin{equation*}
\begin{split}
A_t=&\, A(\sigma_t,b_t^\pm;C_\parallel,C_\perp)\coloneqq b_t^+ R_{\sigma_t}JR_{\sigma_t}^{-1}+b_t^-R_{\sigma_t}^{-1}JR_{\sigma_t}\,, \\ 
\bb_t=&\, \bb(\sigma_t,b_t^\pm;C_\perp)\coloneqq \frac{C_\perp}2\big((b_t^-)^2R_{\sigma_t}^{-1}-(b_t^+)^2R_{\sigma_t}\big)\be_2\,, \\
\tilde\bv^1_t=&\, \tilde\bv^1(\sigma_t,b_t^\pm;C_\perp)\coloneqq -\frac{C_\perp}2\big((b_t^+)^2R_{\sigma_t}+(b_t^-)^2R_{\sigma_t}^{-1}\big)\be_2\,, \\ 
\tilde\bv^{2}_t=&\tilde\bv^{3}_t=\, (0,0,0)^\top,\\
c_t=&c(b_t^\pm;C_\perp)\coloneqq \frac{C_\perp}3[(b_t^+)^3+(b_t^-)^3],\\
w_t=&w(b_t^\pm;C_\perp)\coloneqq \frac{C_\perp}3[(b_t^+)^3-(b_t^-)^3]. 
\end{split}
\end{equation*}
As in the previous sections we can gather together this expressions in a matrix form, and write  these equations as a control system as
\begin{equation}\label{217_3}
\begin{pmatrix}R_{\theta_t}^{-1}\dot\bx_t\\ \dot\theta_t\\ \dot\sigma_t\\ \dot b_t^+\\ \dot b_t^-\end{pmatrix}=
\begin{pmatrix}-\cR_t^{-1}\bv_t^1\\ 1\\ 0\\0\end{pmatrix}u_1+
\begin{pmatrix}-\cR_t^{-1}\bv_t^{2}\\0\\1\\0\end{pmatrix}u_2+
\begin{pmatrix}-\cR_t^{-1}\bv_t^{3}\\0\\0\\1\end{pmatrix}u_3
\eqqcolon\sum_{i=1}^3\bg_i(\sigma_t, b_t^\pm) u_i \,,
\end{equation}
where as in the previous sections $\bv_t^1\coloneqq(\tilde\bv_t^1|w_t)$, $\bv_t^{j}\coloneqq(\tilde\bv_t^{j}|0)$, for $j=2,3$ and $\cR_t^{-1}$ is the inverse of the Grand resistance matrix.
Here, for $i=1, 2, 3$, the maps $u_i\colon[0,T]\to\R{}$ are the control functions identified with the time derivatives of the three shape parameters $b_t^{\pm}$ and $\sigma_t$\,.

\begin{theorem}
System \eqref{217_3} is totally controllable.
\end{theorem}
\begin{proof}
We compute the Lie brackets $\bg_t^4\coloneqq[\bg_t^1,\bg_t^2]$, and $\bg_t^5\coloneqq[\bg_t^1,\bg_t^3]$, $\bg_t^6\coloneqq[\bg_t^2,\bg_t^4]$, at the specific values $\sigma_\star=\pi/4$, $b_\star^\pm=L$, obtaining the fields $\bg_\star^i\coloneqq \Pi^G_e(\bg_i(\sigma_\star,b_\star^\pm))$ (see \eqref{Pi_g}), for $i=4,5,6$,
\begin{equation}\label{218_grow}
\begin{split}
\bg^4_\star= &\, \bigg(\displaystyle- \frac{C_\perp}{2\sqrt{2}(C_\perp+C_\parallel)},
\displaystyle- \frac{C_\perp^2}{\sqrt{2}(C_\parallel+C_\perp)(C_\perp+4C_\parallel)},
\displaystyle\frac{3(C_\perp^2+3C_\perp C_\parallel+4C_\parallel^2)}{2L(C_\parallel+C_\perp)(C_\perp+4C_\parallel)}\bigg)^\top,\\
\bg^5_\star= &\, \bigg(
\displaystyle- \frac{C_\perp}{2\sqrt{2}(C_\perp+C_\parallel)},
\displaystyle\frac{C_\perp^2}{\sqrt{2}(C_\parallel+C_\perp)(C_\perp+4C_\parallel)},
\displaystyle-\frac{3(C_\perp^2+3C_\perp C_\parallel+4C_\parallel^2)}{2L(C_\parallel+C_\perp)(C_\perp+4C_\parallel)}\bigg)^\top,\\
\bg^6_\star= &\, \bigg(
\displaystyle \frac{C_\perp(3C_\perp^2+15C_\perp C_\parallel+20C_\parallel^2)}{2\sqrt{2}(C_\parallel+C_\perp)^2(C_\perp+4C_\parallel)},0,
\displaystyle -\frac{3(C_\perp^2+3C_\perp C_\parallel+4C_\parallel^2)}{2L^2(C_\parallel+C_\perp)(C_\perp+4C_\parallel)}\bigg)^\top.
\end{split}
\end{equation}
To apply Corollary \ref{gait_tot}, we have to verify that the vector fields $\bg^4_\star$\,, $\bg^5_\star$\,, and $\bg^6_\star$ are linearly independent so that they generate the Lie algebra $\fg$.
The linear independence of these vector fields follows from the fact that 
$$
\det(\bg_\star^4|\bg_\star^5|\bg_\star^6)=
\frac{3C_\perp^3(C_\perp^2+3C_\perp C_\parallel+4C_\parallel^2)}{4L^2(C_\parallel+C_\perp)^3(C_\perp+4C_\parallel)^2 }>0 \quad\text{for any $C_\perp$\,, $C_\parallel>0$.}
$$
Therefore, condition \eqref{abovecondition} is fulfilled and the systems is gait controllable by Theorem~\ref{th:gaitcontrollable}.
Moreover note that from equation \eqref{217_3} the matrix $F$ defining the equations of the shape variables is constant and with rank $3$. Therefore we can apply Corollary \ref{gait_tot} to conclude that the system is totally controllable.
\end{proof}

\subsection{Telescopic links} \label{sec:telescopic} 
In this section we consider a two-link swimmer in which each of the links has a telescopic structure, with the two outer sections being connected at the hinge, while the two inner sections, sliding inside the outer ones, are placed on the tip-side of the links. 
This situation might also be seen as each of the arms growing (or disappearing) at an intermediate point. 

To have a clear mathematical description, it is convenient to consider a slight generalization of the framework of Section~\ref{sec:generalmodel}, obtained by enlarging the set describing the reference configuration. More precisely, we now split the body of the swimmer into four intervals: two constant intervals $S^\pm=[0,L]$ for some $L>0$, describing the outer sections, and two variable intervals $\widehat{S}^\pm_t=[\widehat a^\pm_t,0]$ with $\widehat a^\pm_t< 0$ describing the exposed part of the inner sections,  cf.~Figure~\ref{fig:grow_tele}. No deformation occurs, so we set $\gamma_t^\pm=\widehat\gamma_t^\pm\equiv 1$ for every~$\eta$ and~$t$.

We can proceed in an analogous way to the previous examples, obtaining
\begin{align*}
		\bx_t^\pm (\eta)=&\bh_t+\eta\be_t^\pm\,, \qquad \widehat\bx_t^\pm(\eta)=\bh_t+(L+\eta-\widehat a_t^\pm)\be_t^\pm \,,\\
		(\bx_t^\pm)'(\eta)=&(\widehat\bx_t^\pm)'(\eta)\,= \be_t^\pm \,, \\
		\bv(\bx_t^+(\eta))=&\,  
R_{\theta_t}\bigg(R^{-1}_{\theta_t}\dot\bh_t+\eta(\dot\theta_t+\dot\sigma_t) R_{\sigma_t}\be_2\bigg),\\
\bv(\bx_t^-(\eta))=&\, R_{\theta_t}\bigg(R^{-1}_{\theta_t}\dot\bh_t+\eta(\dot\theta_t-\dot\sigma_t) R_{-\sigma_t}\be_2\bigg),\\
\bv(\widehat\bx_t^+(\eta))=&R_{\theta_t}\bigg(R^{-1}_{\theta_t}\dot\bh_t+(L+\eta-\widehat a_t^+)(\dot\theta_t+\dot\sigma_t) R_{\sigma_t}\be_2-\dot{\widehat a}_t^+R_{\sigma_t}\be_1\bigg),\\
\bv(\widehat\bx_t^-(\eta))=&R_{\theta_t}\bigg(R^{-1}_{\theta_t}\dot\bh_t+(L+\eta-\widehat a_t^-)(\dot\theta_t-\dot\sigma_t) R_{\sigma_t}\be_2-\dot{\widehat a}_t^-R_{-\sigma_t}\be_1\bigg),\\
\bF_t^\pm=&\, L R_{\theta_t} \bigg[R_{\pm\sigma_t}J R_{\pm\sigma_t}^{-1}R_{\theta_t}^{-1} \dot\bh_t\mp\frac{C_\perp}2 L(\dot\theta_t\pm\dot\sigma_t) R_{\pm\sigma_t}\be_2 \bigg], \\
\widehat\bF_t^\pm=&-\widehat a_t^\pm\,R_{\theta_t} \bigg[R_{\pm\sigma_t}J R_{\pm\sigma_t}^{-1}R_{\theta_t}^{-1} \dot\bh_t \mp C_\parallel\dot{ \widehat{a}}_t^\pm R_{\pm\sigma_t}\be_1+(2L-\widehat a_t^\pm)\frac{C_\perp}2 (\dot\theta_t\pm\dot\sigma_t) R_{\pm\sigma_t}\be_2 \bigg], \\
 \bM_t^\pm=& \bigg[\mp\frac{C_\perp}2 L^2 (R_{\pm\sigma_t}\be_2){\cdot}(R_{\theta_t}^{-1}\dot\bh_t) +\frac{C_\perp}3 L^3(\dot\theta_t\pm\dot\sigma_t)\bigg]\be_3\,,\\
\widehat \bM_t^\pm=& \bigg[-\frac{C_\perp}2 ((L-\widehat a^\pm_t)^2-L^2) (R_{\pm\sigma_t}\be_2){\cdot}(R_{\theta_t}^{-1}\dot\bh_t) +\frac{C_\perp}3 ((L-\widehat a_t^\pm)^3-L^3)(\dot\theta_t\pm\dot\sigma_t)\bigg]\be_3\,.
\end{align*}
The total balance of forces and torques acting on the swimmer \eqref{G07} become
\begin{equation*}
\begin{cases}
\bF_t=R_{\theta_t}\big[ A_t R_{\theta_t}^{-1}\dot\bh_t + \dot\theta_t\bb_t+\dot\sigma_t\tilde\bv^1_t+\dot {\widehat a}_t^+\tilde\bv^2_t+\dot {\widehat a}_t^-\tilde\bv^3_t\big]=\bzero \,,\\
 M_t=\bb_t\cdot(R_{\theta_t}^{-1}\dot\bh_t)+ c_t\dot\theta_t+w_t\dot\sigma_t=0\,,
\end{cases}
\end{equation*}
where
\begin{equation*}
\begin{split}
A_t=&\, A(\sigma_t,\widehat a^\pm_t;C_\parallel,C_\perp)\coloneqq L [R_{\sigma_t}JR_{\sigma_t}^{-1}+R_{-\sigma_t}JR_{-\sigma_t}^{-1}]-\widehat a^+_tR_{\sigma_t}^{-1}JR_{\sigma_t}-\widehat a^-_tR_{-\sigma_t}^{-1}JR_{-\sigma_t}\,, \\ 
\bb_t=&\, \bb(\sigma_t,\widehat a_t^\pm;C_\perp)\coloneqq \frac{C_\perp}2\big((L^2-\widehat a^-_t(2L-\widehat a_t^-))R_{\sigma_t}^{-1}-(L^2-\widehat a^+_t(2L-\widehat a_t^+))R_{\sigma_t}\big)\be_2\,, \\
\tilde\bv^1_t=&\, \tilde\bv^1(\sigma_t,,\widehat a_t^\pm;C_\perp)\coloneqq -\frac{C_\perp}2\big((L^2+\widehat a^+_t(2L-\widehat a_t^+))R_{\sigma_t}+(L^2-\widehat a^-_t(2L-\widehat a_t^-))R_{\sigma_t}^{-1}\big)\be_2\,, \\ 
\tilde\bv^{2}_t=&\, \tilde\bv^2(\sigma_t,\widehat a_t^+;C_\parallel)\coloneqq- C_\parallel R_{\sigma_t}\be_1\,,\\
\tilde\bv^{3}_t=&\, \tilde\bv^2(\sigma_t,\widehat a_t^-;C_\parallel)\coloneqq C_\parallel R_{-\sigma_t}\be_1\,,\\
c_t=&c(\widehat a_t^\pm;C_\perp)\coloneqq \frac{C_\perp}3[((L-\widehat a_t^+)^3+(L-\widehat a_t^-)^3],\\
w_t=&w(a_t;C_\perp)\coloneqq \frac{C_\perp}3((L-\widehat a_t^+)^3-(L-\widehat a_t^-)^3]. 
\end{split}
\end{equation*}
Gathering together the previous expressions, the equations of motion can be written in matrix form as
\begin{equation}\label{217_4}
\begin{pmatrix}R_{\theta_t}^{-1}\dot\bx_t\\ \dot\theta_t\\ \dot\sigma_t\\ \dot{\widehat a}_t^+\\ \dot {\widehat a}_t^-\end{pmatrix}=
\begin{pmatrix}-\cR_t^{-1}\bv_t^1\\ 1\\ 0\\0\end{pmatrix}u_1+
\begin{pmatrix}-\cR_t^{-1}\bv_t^{2}\\0\\1\\0\end{pmatrix}u_2+
\begin{pmatrix}-\cR_t^{-1}\bv_t^{3}\\0\\0\\1\end{pmatrix}u_3
\eqqcolon\sum_{i=1}^3\bg_i(\sigma_t, \widehat a_t^\pm) u_i \,,
\end{equation}
where as in the previous sections $\bv_t^1\coloneqq(\tilde\bv_t^1|w_t)$, $\bv_t^{j}\coloneqq(\tilde\bv_t^{j}|0)$, for $j=2,3$ and $\cR_t^{-1}$ is the inverse of the Grand resistance matrix.
Here, for $i=1, 2, 3$, the maps $u_i\colon[0,T]\to\R{}$ are the control functions identified with the time derivatives of the three shape parameters $\widehat{a}_t^{\pm}$ and $\sigma_t$\,.

\begin{theorem}
In the regime~\eqref{ratio}, system \eqref{217_4} is totally controllable. 
\end{theorem}
\begin{proof}
We compute the Lie brackets $\bg_t^4\coloneqq[\bg_t^1,\bg_t^2]$, and $\bg_t^5\coloneqq[\bg_t^1,\bg_t^3]$, $\bg_t^6\coloneqq[\bg_t^2,\bg_t^3]$ at the specific values $\sigma_\star=\pi/4$, $\widehat a_\star^\pm=-L$, obtaining the fields $\bg_\star^i\coloneqq \Pi^G_e(\bg_i(\sigma_\star,\widehat a_\star^\pm))$ (see~\eqref{Pi_g}), for $i=4,5,6$,
\begin{equation}\label{218_telescopic}
\begin{split}
\!\!\!\! \bg^4_\star= &\, \bigg(\displaystyle \frac{(C_\perp-C_\parallel)^2}{2\sqrt{2}(C_\perp+C_\parallel)^2},
\displaystyle\frac{2C_\perp^3+23C_\perp^2 C_\parallel+8C_\perp C_\parallel^2-16 C_\parallel^3}{\sqrt{2}(C_\parallel+C_\perp)(C_\perp+4C_\parallel)^2},
\displaystyle-\frac{3(C_\perp^3+9C_\perp^2 C_\parallel+10C_\perp C_\parallel^2+8C_\parallel^3)}{4L(C_\parallel+C_\perp)(C_\perp+4C_\parallel)^2}\bigg)^\top,\\
\!\!\!\! \bg^5_\star= &\, \bigg(
\displaystyle \frac{(C_\perp-C_\parallel)^2}{2\sqrt{2}(C_\perp+C_\parallel)^2},
\displaystyle\frac{(4C_\parallel-C_\perp)(2C_\perp^2+C_\perp C_\parallel-4C_\parallel)}{\sqrt{2}(C_\parallel+C_\perp)(C_\perp+4C_\parallel)^2},
\displaystyle-\frac{3(C_\perp^3+5C_\perp^2 C_\parallel+22C_\perp C_\parallel^2+24C_\parallel^3)}{4L(C_\parallel+C_\perp)(C_\perp+4C_\parallel)^2}\bigg)^\top,\\
\!\!\!\! \bg^6_\star = &\, \bigg(
\displaystyle -\frac{C_\parallel^2}{\sqrt{2}L(C_\parallel+C_\perp)(C_\perp+4C_\parallel)},
\displaystyle- \frac{C_\parallel(3C_\perp+2 C_\parallel)}{2\sqrt{2}(C_\parallel+C_\perp)^2(C_\perp+4C_\parallel)},
\displaystyle \frac{3C_\parallel(C_\perp- C_\parallel)}{8L^2(C_\parallel+C_\perp)(C_\perp+4C_\parallel)}\bigg)^\top.
\end{split}
\end{equation}
To apply Corollary \ref{gait_tot}, we have to verify that the vector fields $\bg^4_\star$\,, $\bg^5_\star,$\,, and $\bg^6_\star$ are linearly independent so that they generate the Lie algebra $\fg$.
An explicit computation yields
$$
\det(\bg_\star^4|\bg_\star^5|\bg_\star^6)=p(C_\perp,C_\parallel)q(C_\perp,C_\parallel)\,,
$$
with
$$
p(C_\perp,C_\parallel):=-\frac{3(C_\perp-C_\parallel)}{16L(C_\parallel+C_\perp)^5(C_\perp+4C_\parallel)^3 }
$$
and, denoting $\kappa=C_\perp/C_\parallel$,
$$
q(\kappa)\coloneqq 5\kappa^7+81\kappa^6-149\kappa^5-1149\kappa^4 -972\kappa^3+1096\kappa^2 +832 \kappa-384 \,.
$$
In the regime~\eqref{ratio}, we have that $p(C_\perp,C_\parallel)>0$ and $q(\kappa)>0$. Thus, the determinant is always different from zero and the linear independence of $\bg^3_\star$\,, $\bg^4_\star$, and $\bg^5_\star$ follows.
Therefore, condition \eqref{abovecondition} is fulfilled and the systems is gait controllable by Theorem~\ref{th:gaitcontrollable}.
Moreover note that from equation \eqref{217_4} the matrix $F$ defining the equations of the shape variables is constant and with rank $3$. Therefore we can apply Corollary \ref{gait_tot} to conclude that the system is totally controllable.
\end{proof}

\begin{remark}
The controllability results obtained in Section~\ref{sec:esempi} can be extended to the case of the $N$-link swimmer; the reasoning follows that of \cite[Theorem~4.2]{MMSZ}.
\end{remark}


\section{Numerical simulations and discussion}\label{sec_num_disc}
In this section, we show numerical simulations of the behavior of some of the brackets computed in Section~\ref{sec:esempi}, illustrating the trajectory obtained by iterating the same control loop several times.
As discussed in \cite{FPZ}, the iteration of a control loop leads to one of two alternative behaviors: either the trajectory is drifting along a certain direction, or it is quasi periodic, remaining in a compact set.
More precisely, each zero-mean control loop 
(thus inducing a gait) produces a displacement, called \emph{geometric phase} \cite{KM,MarsdenRatiu}, in the variable $g$. 
If the control loop is of the type of equations \eqref{loop} or \eqref{loopk}, then the geometric phase is well approximated by the Lie bracket associated with such loops. Thus,  iterating~$N$ times a control loop produces a shift in the position variable $g$ which is, approximately, $N$ times the Lie bracket, cf.~\cite[Eq.~(3.2)]{LibroCoron}. 
Similarly, the composition of two ore more different control loops produces a shift given by the sum of their corresponding geometric phases (or, approximately, of the associated brackets). 
In \cite{AM1997,FPZ} it has been shown that, if the position space $G$ is a non-compact Lie group (as in our case with $G= SE(2)$), the ``prevalent'' behaviour, i.e., drift or quasi-periodicity, observed iterating a control loop, can be predicted only by knowing $G$.  In our models, where $G=SE(2)$, quasi-periodicity is prevalent, meaning that, for almost every control loop, its iteration produces a bounded, quasi-periodic trajectory, while it is necessary to single out specific loops to obtain a trajectory that drifts in a certain direction.

In what follows, we illustrate numerically, for our models, gaits associated with both behaviors, providing explicit expressions of the loops. 
In all simulations, we integrate the equations of motion using the following initial conditions: $x_0=0$, $y_0=0$, $\theta_0=0$, and $\sigma_0=\pi/4$ and drag coefficients ratio $C_\perp/C_\parallel=2$ (see~\eqref{ratio}).
\begin{figure}
\centering
\begin{subfigure}{0.4\textwidth}
    \includegraphics[width=0.8\textwidth]{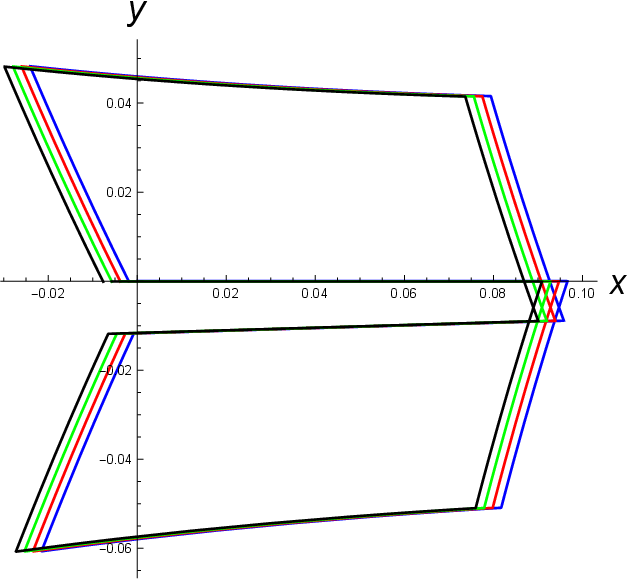}
    \caption{}
    \label{fig:first}
\end{subfigure}
\hfill
\begin{subfigure}{0.4\textwidth}
    \includegraphics[width=0.8\textwidth]{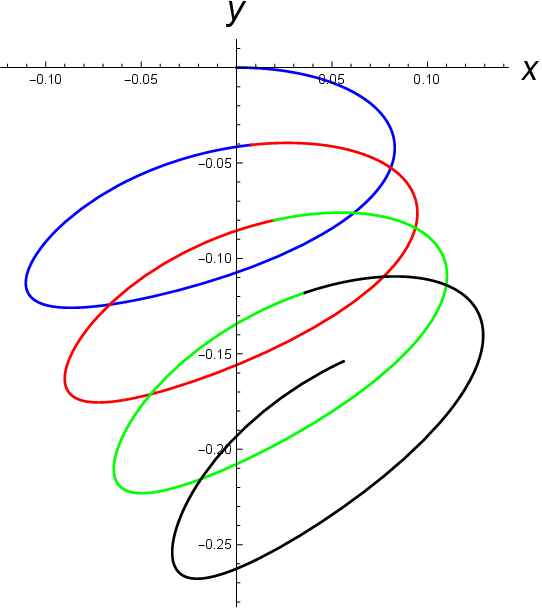}
    \caption{}
    \label{fig:second}
\end{subfigure}
\hfill
\begin{subfigure}{0.4\textwidth}
    \includegraphics[width=0.8\textwidth]{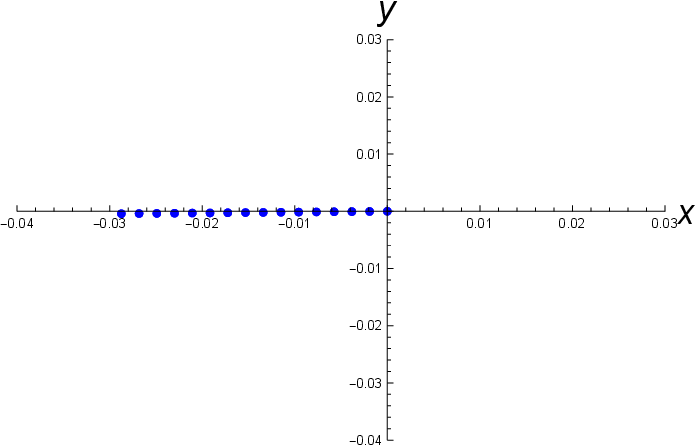}
    \caption{}
    \label{fig:third}
\end{subfigure}
        \hfill
        \begin{subfigure}{0.4\textwidth}
    \includegraphics[width=0.8\textwidth]{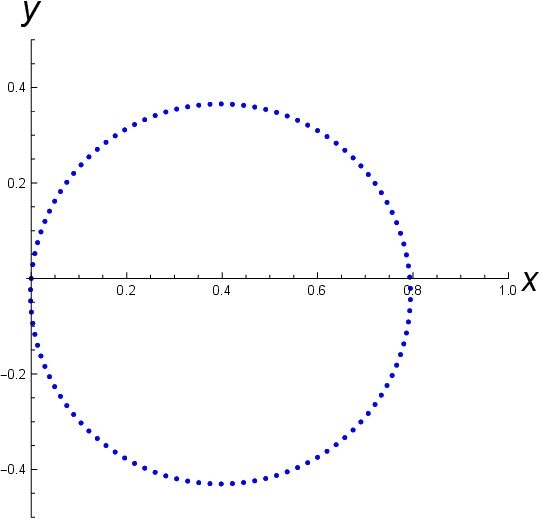}
    \caption{}
    \label{fig:fourth}
\end{subfigure}
\caption{Figure~\ref{fig:first}: trajectory in the $(x,y)$-plane obtained by iterating four times the concatenation of the two control loops  in \eqref{conc_loops}. Figure~\ref{fig:second}: trajectory in the $(x,y)$-plane obtained by iterating four times the control loop \eqref{Rot_loop}. Figure~\ref{fig:third}: $(x,y)$-displacement after each iteration of the concatenation of the two control loops in \eqref{conc_loops}. Figure~\ref{fig:fourth}: $(x,y)$-displacement after each iteration of the control loop \eqref{Rot_loop}}
\label{fig:Stretching_links}
\end{figure}

Considering the stretching links swimmer, we start with initial conditions  $\gamma^+_0=\gamma^-_0=1$ for the deformation gradient. We show in Figure~\ref{fig:first}  the trajectory in the $(x,y)$-plane obtained by iterating four times the concatenation of control loops given in \eqref{loop}; more precisely, recalling~\eqref{217_1}, we set 
\begin{equation}\label{conc_loops}
\tilde{\bu}(t)=\begin{cases}
\be_1 & \text{for $0\leq t<0.1$,}\\
\be_2 & \text{for $0.1\leq t<0.2$,}\\
-\be_1& \text{for $0.2\leq t<0.3$,}\\
-\be_2 & \text{for $0.3\leq t\leq 0.4$,}\\
\end{cases}
\qquad \hat{\bu}(t)=\begin{cases}
\be_2 & \text{for $0\leq t<0.1$,}\\
\be_3 & \text{for $0.1\leq t<0.2$,}\\
-\be_2& \text{for $0.2\leq t<0.3$,}\\
-\be_3 & \text{for $0.3\leq t\leq 0.4$,}
\end{cases}
\end{equation}
where the concatenation is the loop with period $2\tau=0.8$ given by
\begin{equation*}
\tilde{\bu}^\frown \hat{\bu}(t)\coloneqq \begin{cases}
\tilde{\bu}(t) &\text{ for $0\leq t$ mod$(0.8)\leq 0.4$,}\\
\hat{\bu}(t) &\text{ for $0.4< t$ mod$(0.8)\leq 0.8$.}
\end{cases}
\end{equation*}
The concatenation of these two loops results in a shift in the position variable $g$ equal to  $\bg_*^4+\bg_*^5$ (up to order $\tau^2$) in \eqref{218}. Plotting the $(x,y)$-displacement obtained after each iteration of the concatenation (i.e., any $2\tau$) we obtain Figure \ref{fig:third}. As we theoretically expect from the sum of the brackets $\bg_*^4+\bg_*^5$ in \eqref{218}, we obtain a translation in the horizontal direction.

To illustrate instead the quasi periodic behavior, which is the predominant one, we consider, for simplicity, the following control loop
\begin{equation}\label{Rot_loop}
\bu(t)=\begin{pmatrix}
0.1\cos(t)\\
0.1\sin(t)
\end{pmatrix}
\end{equation}
The trajectory in the $(x,y)$-plane obtained iterating four times this loop is shown in Figure~\ref{fig:second}, while Figure~\ref{fig:fourth} illustrates the $(x,y)$-displacement obtained after each iteration, producing a quasi periodic trajectory, as expected.

For the other examples, the strategy is similar: to obtain the drifting behavior, we look for special control loops among the ones that generate a Lie bracket whose rotation component is zero.
For the case of sliding links, recalling~\eqref{217_2}, we obtain a translation along the  $x$-axis by iterating a control loop of the form~\eqref{loopk} associated with the bracket~$\bg_*^6$ in~\eqref{218_slide}. 
In Figure~\ref {fig:slide1}, we plot the $(x,y)$-displacement after each iteration. 
The small deviation from the $x$-axis is due to higher order effects: indeed the bracket $\bg_*^6$ is only an approximation of the geometric phase. 
For the case of growing links, recalling~\eqref{217_3}, the iterated concatenation of two control loops of the form \eqref{loop} associated with the Lie brackets~$\bg_*^4$ and~$\bg_*^5$ in~\eqref{218_grow} leads to a translation along the horizontal axis (see Figure~\ref{fig:grow1}). 
Finally, for the telescopic case, since the brackets $\bg_*^4$,~$\bg_*^5$, and~$\bg_*^6$ in~\eqref{218_telescopic} are linearly independent, it is possible to find a suitable concatenation of them which results in a drifting behavior.

In all examples, we can easily obtain the quasi-periodic trajectory  since this is the prevalent type, for example taking the control loop \eqref{Rot_loop} (see Figures \ref{fig:slide2} and \ref{fig:grow2}).
\begin{figure}
\centering
\begin{subfigure}{0.4\textwidth}
    \includegraphics[width=0.8\textwidth]{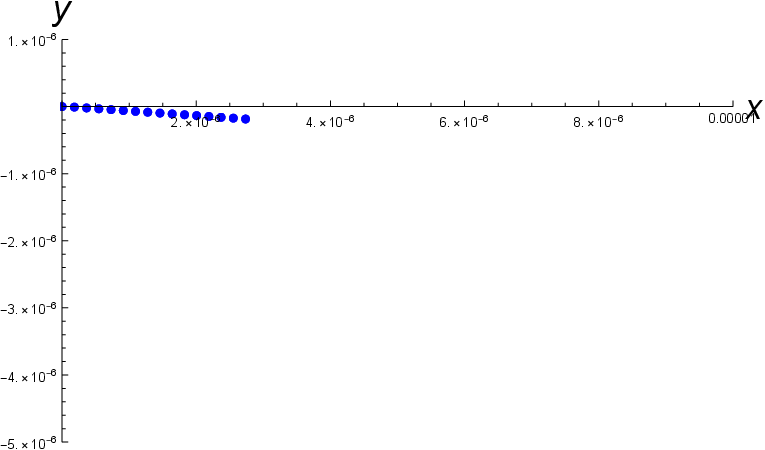}
    \caption{}
    \label{fig:slide1}
\end{subfigure}
\hfill
\begin{subfigure}{0.4\textwidth}
    \includegraphics[width=0.8\textwidth]{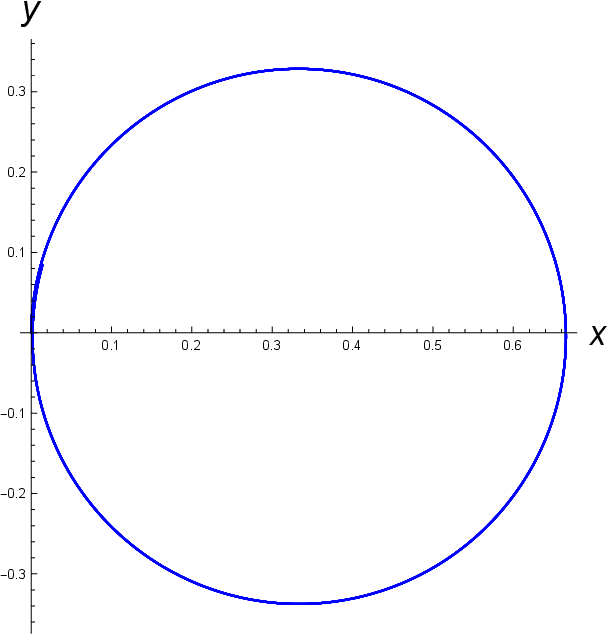}
    \caption{}
    \label{fig:slide2}
\end{subfigure}
\caption{Iteration of control loops in the sliding case. The integration of equations \eqref{217_2} is performed with  $a(0)=-1/2$. Figure~\ref{fig:slide1}: $(x,y)$-displacement after each iteration of the control loop \eqref{loopk} with $A=B=1$ and $\tau=0.04$. Figure~\ref{fig:slide2}: $(x,y)$ displacement after each iteration of the control loop \eqref{Rot_loop}.}
\label{fig:Sliding_links}
\end{figure}
\begin{figure}
\centering
\begin{subfigure}{0.4\textwidth}
    \includegraphics[width=0.8\textwidth]{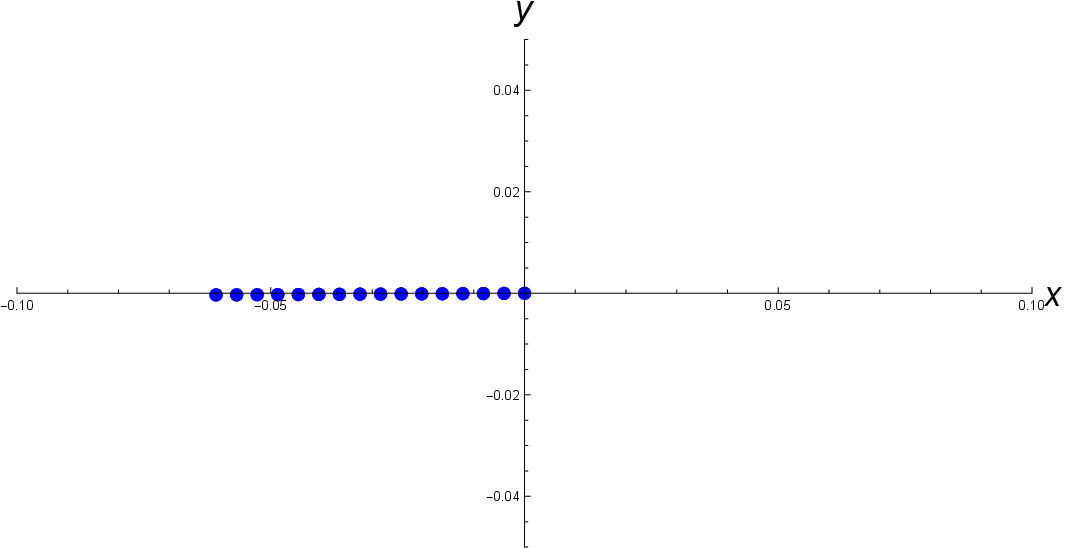}
    \caption{}
    \label{fig:grow1}
\end{subfigure}
\hfill
\begin{subfigure}{0.4\textwidth}
    \includegraphics[width=0.8\textwidth]{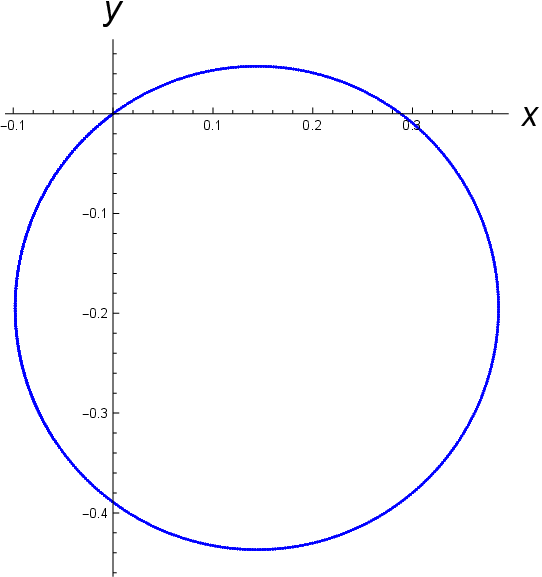}
    \caption{}
    \label{fig:grow2}
\end{subfigure}
\caption{Iteration of control loops in the growing case. The integration of equations \eqref{217_3} is performed with $b^\pm(0)=1$. Figure~\ref{fig:grow1}: $(x,y)$-displacement after each iteration of the concatenation of the two control loops in \eqref{conc_loops}. Figure~\ref{fig:grow2}: $(x,y)$-displacement after each iteration of the control loop \eqref{Rot_loop}.}
\label{fig:Growing_links}
\end{figure}

\subsection*{Acknowledgements} 
PG and MZ are members of the \emph{Gruppo Nazionale di Fisica Matematica} of the \emph{Istituto Nazionale di Alta Matematica}. MM is a member of the \emph{Gruppo Nazionale per l'Analisi Matematica, la Probabilit\`{a} e le loro Applicazioni} of the \emph{Istituto Nazionale di Alta Matematica}.

PG acknowledges the GAČR Junior Star Grant 21-09732M.
This study was carried out within the \emph{Mathematics for Industry 4.0} 2020F3NCPX PRIN2020 (MM and MZ) funded by the Italian MUR,  the  \emph{Geometric-Analytic Methods for PDEs and Applications} 2022SLTHCE (MM) and the \emph{Innovative multiscale approaches, possibly based on Fractional Calculus, for the effective constitutive modeling of cell mechanics, engineered tissues, and metamaterials in Biomedicine and related fields} P2022KHFNB (MZ) projects funded by the European Union -- Next Generation EU  within the PRIN 2022 PNRR program (D.D. 104 - 02/02/2022). 
This manuscript reflects only the authors’ views and opinions and the Ministry cannot be considered responsible for them.

\bibliographystyle{plain}

\begin{thebibliography}{100}


\bibitem{LibroAgrachev} A.~A.~Agrachev and Y.~L.~Sachkov: \emph{Control theory from the geometric viewpoint}, Encyclopaedia of Mathematical Sciences, \textbf{87}. Control Theory and Optimization, II. Springer-Verlag, Berlin, 2004.

\bibitem{ADSGZ} F.~Alouges, A.~DeSimone, L.~Giraldi, and M.~Zoppello: \emph{Self-propulsion of slender micro-swimmers by curvature control: $N$-link swimmers}. International Journal of Non-Linear Mechanics \textbf{56}, 132--141.

\bibitem{AM1997} P.~Ashwin and I.~Melbourne: \emph{Noncompact drift for relative equilibria and relative periodic orbits}. Nonlinearity \textbf{10} (1997), 595--616.

\bibitem{AZN} R.~Attanasi, M.~Zoppello, G.~Napoli: \emph{Purcell's swimmers in pairs}. Physical Review E \textbf{109}(2) (2024), 024601.

\bibitem{AKO05} J.~E.~Avron, O.~Kenneth, and D.~H.~Oaknin: \emph{Pushmepullyou: an efficient micro-swimmer}. New Journal of Physics \textbf{7} (2005), 234.

\bibitem{BM2024} A.~S.~Boiardi and R.~Marchello: \emph{Breaking the left‑right symmetry in futtering artifcial cilia that perform nonreciprocal oscillations}. Meccanica \textbf{59} (2024).

\bibitem{BCHOH} L.~H.~Blumenschein, M.~M.~Coad, D.~A.~Haggerty, A.~M.~Okamura, and E.~W.~Hawkes, \emph{Design, modeling, control, and application of everting vine robots}. Frontiers in Robotics and AI \textbf{7} (2020), 548266.

\bibitem{CM2011} T.~Chambrion and A.~Munnier: \emph{Locomotion and control of a self-propelled shape-changing body in a fluid}. J. Nonlinear Sci. \textbf{21}(3) (2011), 325--385.

\bibitem{CDN2023} G.~Cicconofri, V.~Damioli, and G.~Noselli: \emph{Nonreciprocal oscillations of polyelectrolyte gel filaments subject to a steady and uniform electric field}. J. Mech. Phys. Solids \textbf{173} (2023), Paper No. 105225.

\bibitem{LibroCoron} J.-M.~Coron: \emph{Control and nonlinearity}. Math. Surveys Monogr., \textbf{136}. American Mathematical Society, Providence, RI, 2007.

\bibitem{CMOZ} J.~Cort\'{e}s, S.~Mart\'{i}nez, J. P.~Ostrowski and H.~Zhang: \emph{Simple Mechanical Control Systems with Constrains and Symmetry}.
SIAM J. Control Optim. \textbf{41} (2002) 851--874.

\bibitem{DeSTat} A.~DeSimone and A.~Tatone: \emph{Crawling motility through the analysis of model locomotors: Two case studies}. Eur.~Phys.~J.~E (2012) \textbf{35}: 85.

\bibitem{ECS2023} M.~Eberhard, A.~Choudhary, and H.~Stark: \emph{Why the reciprocal two-sphere swimmer moves in a viscoelastic environment}. Physics of Fluids \textbf{35}(6) (2023), 063119.

\bibitem{FPZ}F.~Fass\`{o}, S.~Passarella and M.~Zoppello: \emph{Control of locomotion systems and dynamics in relative periodic orbits}. J. Geom. Mech. \textbf{12} (2020), 395--420.

\bibitem{FRKHJ} B.~M.~Friederich, I. H.~Riedel-Kruse, J.~Howard, and F. J\"{u}licher: \emph{High-precision tracking of sperm swimming fine structure provides strong test of Resistive Force Theory}. Journal of Experimental Biology \textbf{213} (2010), 1226--1234.

\bibitem{GMZ} L.~Giraldi, P.~Martinon, and M.~Zoppello: \emph{Controllability and optimal strokes for $N$-link microswimmer}. 52nd IEEE Conference on Decision and Control, 3870--3875.

\bibitem{GH1955} J.~Gray and G.~J.~Hancock: \emph{The Propulsion of Sea-Urchin Spermatozoa}. Journal of Experimental Biology \textbf{32}(1955), 802--814.

\bibitem{Hale} J.~K.~Hale: \emph{Ordinary differential equations}. Second edition. Robert E. Krieger Publishing Co., Inc., Huntington, NY, 1980.

\bibitem{HLMS} W.~Hu1, G.~Z.~Lum, M.~Mastrangeli and M.~Sitti: \emph{Small-scale soft-bodied robot with multimodal locomotion}. Nature \textbf{554} (2018), 81--85.

\bibitem{LibroJurdjevic} V.~Jurdjevic: \emph{Geometric Control Theory}. Cambridge University Press, Cambridge, 1997.

\bibitem{KM} S. D.~Kelly and R. M.~ Murray: \emph{Geometric phases and robotic locomotion}. J. Field Robotics \textbf{12} (1995), 417--431.

\bibitem{KTSKS} I.~S.~M.~Khalil, A.~F.~Tabak, M.~A.~Seif, A.~Klingner, M.~Sitti: \emph{Controllable switching between planar and helical flagellar swimming of a soft robotic sperm}. PLoS ONE \textbf{13}(11) (2018), e0206456.

\bibitem{KoensAl} L.~Koens, H.~Zhang, M.~M\"{o}ller, A.~Mourran, and E.~Lauga: \emph{The swimming of a deforming helix}. The European Physical Journal E \textbf{41} (2018), 1--14.

\bibitem{KropacekAl} J.~Kropacek, C.~Maslen, P.~Gidoni, P.~Cigler, F.~Stepanek and I.~Rehor: \emph{Light-Responsive Hydrogel Microcrawlers, Powered and Steered with Spatially Homogeneous Illumination}, Soft Robotics, online first.

\bibitem{Lauga_book} E.~Lauga: \emph{The fluid dynamics of cell motility}. Cambridge Texts in Applied Mathematics, Cambridge University Press, Cambridge, 2020.

\bibitem{GJ2020} L.~Giraldi, F.~Jean: \emph{Periodical body deformations are optimal strategies for locomotion}, SIAM J. Control Optim. \textbf{58}(3) (2020),700--1714.

\bibitem{MKL} Y.~Man, L.~Koens, and E.~Lauga: \emph{Hydrodynamic interactions between nearby slender filaments}. Europhysics Letters, \textbf{116}(2) (2016), 24002.

\bibitem{MMSZ} R.~Marchello, M.~Morandotti, H.~Shum and M.~Zoppello: \emph{The $N$-link swimmer in three dimensions: controllability and optimality results}. Acta Appl. Math. \textbf{178} (2022).

\bibitem{MarsdenRatiu} J.~E.~Marsden and T.~S.~Ratiu: \emph{Introduction to Mechanics and Symmetry}. Texts Appl. Math., 17, Springer-Verlag, New York, 1999.

\bibitem{MedAl} M.~Medina-S\'{a}nchez, V.~Magdanz, M.~Guix, V.~M.~Fomin,
and O.~G.~Schmidt: \emph{Swimming microrobots: Soft, reconfigurable, and smart}. Advanced Functional Materials \textbf{28} (2018), 1707228.

\bibitem{MourranAl17} A.~Mourran, H.~Zhang, R.~Vinokur, and M.~M\"{o}ller: \emph{Soft microrobots employing nonequilibrium actuation via plasmonic heating}. Advanced Materials, \textbf{29}(2) (2017), 1604825.

\bibitem{MourranAl21} A.~Mourran, O.~Jung, R.~Vinokur, and M.~M\"{o}ller: \emph{Microgel that swims to the beat of light}. The European Physical Journal E, \textbf{44}(6) (2021), 79.

\bibitem{OAR} A.~Oulmas, N.~Andreff, and S.~Régnier: \emph{3D closed-loop swimming at low Reynolds numbers}. The International Journal of Robotics Research \textbf{37}(11) (2018), 1359--1375.

\bibitem{Purcell77} E.~M.~Purcell: \emph{Life at low Reynolds number}. American Journal of Physics \textbf{45} (1977), 3--11.

\bibitem{SMM17}  A.~Sadeghi, A.~Mondini and B.~Mazzolai: \emph{Toward Self-Growing Soft Robots Inspired by Plant Roots and Based on Additive Manufacturing Technologies}. Soft Robotics \textbf{4} (2017), 211--223. 

\bibitem{SharanAl21}  P.~Sharan, C.~Maslen, B.~Altunkeyik, I.~Rehor, J.~Simmchen and T.~D.~Montenegro-Johnson: \emph{Fundamental Modes of Swimming Correspond to Fundamental Modes of Shape: Engineering I-, U-, and S-Shaped Swimmers}. Advanced Intelligent Systems  \textbf{3}(11) (2021), 2100068.

\bibitem{SG2012} H.~Shum and E.~A.~Gaffney: \emph{The effects of flagellar hook compliance on motility of monotrichous bacteria: A modeling study}. Physics of Fluids \textbf{24} (2012), 061901.

\bibitem{RehorAl21} I.~Rehor, C.~Maslen, P.~G.~Moerman, B.~G.~P.~Van Ravensteijn, R.~Van Alst, J.~ Groenewold, H.~B.~Eral and W. K. Kegel: \emph{Photoresponsive hydrogel microcrawlers exploit friction hysteresis to crawl by reciprocal actuation}. Soft Robotics, \textbf{8}(1) (2021), 10--18.

\bibitem{TanAl}I.~Tanasijevi\'{c}, O.~Jung, L.~Koens, A.~Mourran, and E.~Lauga: \emph{Jet-driven viscous locomotion of confined thermoresponsive microgels}. Applied Physics Letters \textbf{120}(10) (2022), 104101.

\bibitem{Taylor1951} G.~I.~Taylor: \emph{Analysis of the swimming of microscopic organisms}. Proc.~R.~Soc.~Lond.~A \textbf{209} (1951), 447--461.

\bibitem{UHSOHF} N.~S.~Usevitch, Z.~M.~Hammond, M.~Schwager, A.~M.~Okamura, E.~W.~Hawkes, and S.~Follmer, \emph{An untethered isoperimetric soft robot}. Science Robotics \textbf{5} (2020), eaaz0492.

\bibitem{Wang2019} Q.~Wang: \emph{Optimal Strokes of Low Reynolds Number Linked-Sphere Swimmers}. Appl. Sci. \textbf{9}(19) (2019), 4023.

\bibitem{ZMBG} M.~Zoppello, M.~Morandotti, and H.~Bloomfield-Gad\^{e}lha: \emph{Controlling non-controllable scallops}. Meccanica \textbf{57}(9) (2022), 2187--2197.

\end{thebibliography}

\end{document}